\documentclass[]{article}

\usepackage[a4paper, total={6in, 8in}]{geometry}
\usepackage[utf8]{inputenc}
\usepackage{graphicx,xcolor,float,enumitem,mathtools}
\usepackage{amsmath,amssymb,bm,pgfplots}
\usepackage[sort,compress]{cite}
\usepackage[capitalise,nameinlink]{cleveref}
\usepackage{amsthm}
\usepackage{authblk}

\pgfplotsset{compat=1.18}

\crefname{subsection}{Subsection}{Subsections}
\crefformat{equation}{(#2#1#3)}
\crefformat{align}{(#2#1#3)}
\crefformat{enumi}{(#2#1#3)}
\crefname{figure}{Figure}{Figures}

\newtheorem{theorem}{Theorem}[section]
\newtheorem{definition}{Definition}[section]
\newtheorem{remark}{Remark}[section]
\newtheorem{lemma}{Lemma}[section]
\newtheorem{proposition}{Proposition}[section]

\newtheorem{example}{Example}
\newtheorem{assumption}{Assumption}

\allowdisplaybreaks

\newcommand{\R}{\mathbb{R}}
\newcommand{\eps}{\varepsilon}

\newcommand{\N}{\mathbb{N}}
\newcommand{\dt}{\,\textup{d}t}
\newcommand{\ds}{\,\textup{d}s}
\newcommand{\ddt}{\frac{\textup{d}}{\textup{d}t}}
\newcommand{\dz}{\,\textup{d}z}
\newcommand{\dx}{\,\textup{d}x}
\newcommand{\dy}{\,\textup{d}y}
\newcommand{\dxy}{\,\textup{d}(x,y)}

\usepackage{array}
\newcolumntype{M}[1]{>{\centering\arraybackslash}m{#1}}

\DeclarePairedDelimiter{\abs}{\lvert}{\rvert}
\DeclarePairedDelimiter{\norm}{\lVert}{\rVert}

\usepackage{comment,float,graphicx}

\title{Unifying local and nonlocal corrosion frameworks:\\ A convergent nonlocal extension of the KKS phase-field model}
\author{Christian J. Cyron$^{1,2}$,\;Marvin Fritz$^3$,\;Alexander Hermann$^{1,4,\ast}$,\\\;Tobias K\"oppl$^5$,\;Arman Shojaei$^1$,\;Stewart Silling$^6$}
\date{%
	$^1$Institute of Material Systems Modeling, Helmholtz-Zentrum Hereon, \\Geesthacht Germany\\ [2ex]
	$^2$Institute for Continuum and Material Mechanics, Hamburg University of Technology,  Hamburg, Germany\\ [2ex]
	$^3$Johann Radon Institute for Computational and Applied Mathematics, \\Linz, Austria \\ [2ex]
	$^4$CAU Innovation GmbH, Kiel, Germany  \\ [2ex]
	$^5$Fraunhofer-Institut FOKUS, Berlin, Germany \\ [2ex]
	$^6$Nearforce LLC, Albuquerque, NM, USA \\ [2ex]
	$^\ast$Correspondence: alexander.hermann@hereon.de \\ [2ex]
	\today
}
\begin{document}

\maketitle

\begin{abstract}
We introduce a nonlocal extension of the Kim-Kim-Suzuki (KKS) phase-field corrosion model aimed at bridging local and nonlocal corrosion modeling approaches, such as phase-field and peridynamic frameworks. In this formulation, classical gradient operators are replaced with integral operators defined over a finite interaction horizon, naturally embedding an intrinsic length scale that aligns with nonlocal theories like peridynamics. Under precise assumptions on function spaces and kernel functions, we define a nonlocal free energy that integrates a standard bulk free energy density with a nonlocal interaction term. Through differentiation in an appropriate Hilbert space, we derive evolution equations, yielding a nonlocal Allen--Cahn equation for the phase-field and a nonlocal Cahn--Hilliard-type equation for the concentration. The latter is expressed as a gradient flow in a metric induced by the inverse of the nonlocal operator, mirroring the classical \(H^{-1}\) metric for conserved dynamics. We establish the well-posedness of these equations using Galerkin approximations, uniform energy estimates, and compactness arguments. Furthermore, we prove the convergence of the nonlocal model to its local KKS counterpart as the interaction horizon approaches zero, effectively unifying local and nonlocal perspectives. Numerical experiments, implemented via finite difference spatial discretization and explicit time-stepping, demonstrate the effects of nonlocality and confirm the theoretical convergence, reinforcing the connection between the two modeling paradigms.
\end{abstract}
\ \\
\textbf{keywords:} phase-field, nonlocal, corrosion, well-posedness, nonlocal-to-local convergence, explicit scheme, peridynamics, Allen--Cahn/Cahn--Hilliard system, Kim--Kim--Suzuki phase-field model

\section{Introduction}
\label{sec:introduction}

Nonlocal phase-field models have attracted significant attention in scientific computing because they naturally incorporate long-range interactions and introduce an intrinsic length scale into the governing equations \cite{silling2000reformulation}. Unlike classical local formulations that employ differential operators---which capture only infinitesimal interactions---nonlocal formulations replace these operators with integral operators that account for interactions over a finite horizon. This replacement not only provides a natural regularization mechanism for evolving interfaces but also establishes a direct connection with peridynamic theories, where the interaction horizon is a fundamental parameter \cite{silling2000reformulation}.

In many applications, including phase transitions \cite{choksi2009phase} and corrosion phenomena (e.g., in biodegradable magnesium implants \cite{hermann2025nonlocal,mai2016phase}), classical local models fail to capture important multiscale effects due to the absence of an intrinsic length scale. For instance, in the two-phase metal--electrolyte setting of corrosion, the classical KKS (Kim--Kim--Suzuki) phase-field model \cite{kim1999phase}---with a free energy comprising a bulk term and a gradient term---does not explicitly incorporate a finite spatial interaction range. Recent studies have demonstrated that introducing nonlocal interactions via integral operators can significantly alter interface dynamics, sometimes even producing phenomena that are not observed in local theory \cite{burkovska2021nonlocal,ren2021nonlocal,fritz2019local}.

Peridynamic models, initially proposed by Silling \cite{silling2000reformulation}, are reformulations of
classical local continuum mechanics in the form of nonlocal integral differential
equations. With the spatial derivatives of the stress tensor replaced by integrals of
force density functions, a nonlocal model allows the natural treatment of balance
laws on and off material discontinuities. We refer to the overview articles \cite{silling2007peridynamic,silling2010peridynamic} and the books \cite{bobaru2016handbook,madenci2013peridynamic} on perdiynamic modeling.

Motivated by these developments and rigorous convergence results in nonlocal-to-local settings of Cahn--Hilliard and Allen--Cahn models \cite{abels2024strong,davoli2021nonlocal,elbar2023degenerate,trussardi2019nonlocal}, we present in this work a rigorous nonlocal formulation of the KKS phase-field corrosion model. Our approach is based on defining a free energy functional that couples a standard bulk free energy density with a nonlocal interaction term characterized by a finite cut-off radius, akin to those employed in peridynamic models \cite{silling2000reformulation} and in tumor growth models with long-range interaction radius \cite{fritz2023tumor,fritz2019local}. We rigorously derive the associated evolution equations via differentiation in a suitable Hilbert space framework. Our derivation yields a nonlocal Allen--Cahn equation for the phase-field and a nonlocal Cahn--Hilliard-type equation for the concentration, the latter formulated as a gradient flow in a metric induced by the inverse of the nonlocal operator, analogous to the classical \(H^{-1}\)-metric that naturally enforces mass conservation.

While our analysis, based on Galerkin approximations, uniform energy estimates, and compactness arguments, establishes the well-posedness of both equations and demonstrates that, under standard smoothness assumptions and proper scaling, the nonlocal model converges to its classical local counterpart, our primary focus is on clarifying the impact of the finite interaction horizon on the solution's behavior. Furthermore, we highlight a structural equivalence between the nonlocal KKS model and peridynamic corrosion models under specific parameter regimes, thereby bridging phase-field and peridynamic approaches to corrosion modeling. 

The remainder of the paper is organized as follows: In Section~\ref{sec:math-framework} we introduce the mathematical framework and state our assumptions. Section~\ref{sec:evolution} reviews the classical local KKS phase-field corrosion model. Furthermore, we introduce the nonlocal free energy functional and derive the corresponding evolution equations via  differentiation. In Section~\ref{sec:wellposed_analysis}, we prove the existence and uniqueness of the decoupled and fully coupled nonlocal models.  Section~\ref{sec:limit} discusses the nonlocal-to-local convergence of the derived operators, while Section~\ref{sec:discretization_and_experiments} presents our finite difference spatial discretization and explicit time-stepping scheme and presents numerical experiments on a two-dimensional pitting corrosion benchmark, and Section~\ref{sec:conclusion} concludes with a discussion and outlook for future work.

\section{Mathematical framework and standing assumptions}
\label{sec:math-framework}
In the following, we state the notational conventions and assumptions that are used throughout the paper. In our work, \(\Omega\subset\mathbb{R}^d\) (\(d\in\{1,2,3\}\)) is a bounded Lipschitz domain and \(T>0\) is a fixed final time. The mean-free space of square-integrable functions is denoted by $\mathring{L}^2(\Omega)$. Moreover, we write $\lesssim$ instead of $\leq C$ when omitting a generic constant $C$, which may also change from line to line.
\begin{assumption}[Kernel properties]
	\label{assump:kernel}
	The kernel $J_\delta: \mathbb{R}^d \to \mathbb{R}$, $\delta>0$, satisfies:\\[-0.5em]
	\begin{enumerate}[label=\textup{(J\arabic*)}, ref=J\arabic*, leftmargin=.9cm] 
		\item \textbf{Nonnegativity}: $J_\delta(z) \geq 0$ for all $z \in \mathbb{R}^d.$ \label{assump:nonneg}
		\item \textbf{Radial symmetry}: $J_\delta(z) = j_\delta(\abs{z})$ for some $j_\delta: [0,\infty) \to \mathbb{R}.$
		\item \textbf{Compact support}: $\textup{supp}(J_\delta) \subseteq B_\delta(0).$
		\item \textbf{Embeddings}: \label{assump:comp}
		$V_\delta \subset L^4(\Omega)$ and $V_\delta \Subset L^2(\Omega)$ where 
		$$V_\delta = \Bigl\{ u \in L^2(\Omega) : 
		|u|_{{V_\delta}}^2 := \frac12 \iint_{\Omega \times \Omega} J_{\delta}(x-y) [u(x)-u(y)]^2  \dxy < \infty \Bigr\}.
		$$
	\end{enumerate}
\end{assumption}
\ \\
Before mentioning relevant examples of kernels satisfying these stated properties, we state one more assumption on the kernel function $J_\delta$ that becomes relevant when considering the nonlocal-to-local convergence $\delta \to 0$ in Section~\ref{sec:limit}.
\begin{assumption}[Additional kernel property]\label{assump:kernel2}
	The kernel $J_\delta: \mathbb{R}^d \to \mathbb{R}$, $\delta>0$, satisfies:\\[-0.5em]
	\begin{enumerate}[label=(J\arabic*), ref=J\arabic*, start=5, leftmargin=.9cm] 
		\item \textbf{Second-moment normalization}:$
		\int_{\mathbb{R}^d} J_\delta(z)\abs{z}^2  \dz = 2d$. \label{assump:2ndnorm}
	\end{enumerate}
\end{assumption}

\begin{example} \label{Ex:Kernels}
	The stated properties \eqref{assump:nonneg}--\eqref{assump:2ndnorm} are satisfied by: 
	\begin{itemize}
		\item the top-hat kernel  \begin{equation} \label{Eq:NormalizationJdelta}
			J_\delta(z) = c_\delta \chi_{B_\delta(0)}(z) \quad \forall z \in \mathbb{R}^d, \quad
			c_\delta= \frac{2d(d+2)}{\delta^{d+2} |S^{d-1}|},
		\end{equation}
		where \(S^{d-1}\) denotes the surface of the unit sphere in \(\mathbb{R}^d\);
		\item the truncated fractional kernel \cite[Eq.~(2.7)]{du2023nonlocal}
		\[    J_\delta(z)=c_\delta\abs{z}^{-d-2s}\chi_{B_\delta(0)}(z) \quad \forall z \in \mathbb{R}^d, \quad d/2 < s < 2.
		\]
		where $c_\delta$ is chosen as in \cite[Eq.~(2.20)--(2.22)]{du2023nonlocal}, depending on the size of $\delta$, such that $J_\delta \to -\Delta$ as $\delta \to 0$ and $J_\delta \to (-\Delta)^s$ as $\delta \to \infty$.
	\end{itemize}
	The first three properties of Assumption~\ref{assump:kernel} are easy to check for the top-hat and truncated fractional kernels. Regarding the compact embedding \eqref{assump:comp}, we refer to \cite[Thm.~1.9]{scott2024nonlocal} and \cite[Cor~3.12]{du2019nonlocal}.   It is well-known for these examples that it holds $J_\delta u \to -\Delta u$ as $\delta \to 0$, see \cite{du2019nonlocal,du2012analysis} for further details. In this regard, it also holds that $\mathring{V}_\delta \to \mathring{H}^1$ as $\delta \to 0$. Furthermore, in these examples, there is a nonlocal Poincaré inequality of the form $$\|u\|_{L^2} \leq C_P(\delta) |u|_{V_\delta}^2 \qquad \forall u \in \mathring{V}_\delta,$$ 
	that is proved in \cite[Prop.~4.1]{du2012analysis}, \cite[Lem.~3.7]{gunzburger2010nlvc}, \cite[Lem.~3.2]{du2023nonlocal}, \cite[Prop.~3.22]{du2019nonlocal}. In such a case, the seminorm $|\cdot|_{V_\delta}$ is an equivalent norm on $\mathring{V}_\delta$ in the sense of
	\begin{equation*} 
		|u|_{V_\delta}^2 = \langle \mathcal{L}_\delta u,u \rangle \geq \kappa(\delta) \|u\|_{V_\delta}^{2}:=\kappa(\delta) \big( \|u\|_{L^{2}(\Omega)}^{2}+|u|_{{V_\delta}}^{2} \bigr) \quad \forall u \in \mathring{V}_\delta,
	\end{equation*}
	where $\kappa(\delta)=(1+C_P(\delta))^{-2}$.
\end{example}
\ \\
The seminorm on $V_\delta$ naturally induces the semi-inner product
\begin{equation}
(u,v)_{{V_\delta}} 
:=\frac12 \iint_{\Omega \times \Omega} J_{\delta}(x-y) [u(x)-u(y)] [v(x)-v(y)]  \dxy \quad \forall u,v \in V_\delta,
\label{eq:V_nl_inner}
\end{equation}
and we define the corresponding nonlocal operator \(\mathcal{L}_\delta:V_\delta \to V_\delta^*\) by 
$$\langle  \mathcal{L}_\delta u, v\rangle=(u,v)_{{V_\delta}} \quad \forall u,v \in V_\delta,$$
using the Riesz representation theorem applied to the bilinear form \((\cdot,\cdot)_{V_\delta}\). Due to the Cauchy--Schwarz inequality
\(
|(u,v)_{V_\delta}| \le |u|_{{V_\delta}}\abs{v}_{{V_\delta}}\) for any $u,v \in V_\delta$, it holds that 
\[ \|\mathcal L_\delta u\|_{V_\delta^{*}} \le |u|_{{V_\delta}}.
\]
This confirms that \(\mathcal{L}_\delta u\) is a bounded linear functional on \(V_\delta\) even when \(u\) is not mean-free.
The nonlocal operator $\mathcal L_\delta$ has been treated in \cite[Thm.~4.4 \& discussion preceding §4.4]{du2012analysis} (fractional, after extending to $\R^{d}$) and \cite[§6, eqs.,(6.2)-(6.3)]{gunzburger2010nlvc} (top-hat). This operator is self-adjoint, positive semi-definite and it is easy to verify that $$\langle \mathcal L_\delta u,1 \rangle=(u,1)_{V_\delta} = 0 \quad \forall u \in V_\delta,$$ that is, $\mathcal L_\delta$ vanishes on constant functions and naturally embeds into $\mathring{V}_\delta^*$.

\begin{assumption}[Free energy density properties] \label{assump:free_energy_density}
The local free energy density $f\in C^{1}(\mathbb{R}^2;\R)$ satisfies:\\[-0.5em]
\begin{enumerate}[label=\textup{(F\arabic*)}, ref=F\arabic*, leftmargin=.9cm]  
	\item \label{assump:splitting} \textbf{Splitting}: $\exists f_1 \in C^{1}(\mathbb{R}^2;\R), f_2 \in C^{1}(\mathbb{R};\R_{\geq 0})$ such that, for any $\phi,c \in \R$,
	$$f(\phi,c)=f_1(\phi,c)+f_2(\phi).$$
	\item \label{assump:affine} \textbf{Affine linearity}: $\exists A>0$, $h_1 \in C^{1,1}(\R)$ with $h_1(0)=0$, $h_2 \in C_b(\R)$, $h_3 \in C(\R)$ with $|h_3(x)| \lesssim 1+\abs{x}$ (linear growth) for any $x\in\R$ such that, for any $\phi,c \in \R$, 
	$$\partial_c f(\phi,c)=\partial_c f_1(\phi,c)=Ac+h_1(\phi), \qquad \partial_\phi f_1(\phi,c)=ch_2(\phi)+h_3(\phi)$$
	\item \label{assump:growth} \textbf{Growth condition}: $\exists C_1,C_2,C_3>0$ such that, for any $\phi,c \in \R$,
	$$\begin{aligned}
		|\partial_\phi f_1(\phi,c)| &\leq C_1 (1+\abs{c}+\abs{\phi}), \quad |\partial_\phi f_2(\phi)| \leq C_2 (1+\abs{\phi}^3), \quad
		|\partial_c f(\phi,c)| \leq C_3 (1+\abs{c}+\abs{\phi});
	\end{aligned}$$
	\item \label{assump:semi} \textbf{Semicoercivity}: $\exists C_4,C_5,C_6,C_7>0$ such that, for any $\phi,c \in \R$,
	$$
	\begin{aligned}
		\bigl(\partial_c f(\phi,c), c \bigr)_{V_\delta} &\geq C_4 \abs{c}_{{V_\delta}}^{2} - C_5 \abs{\phi}_{{V_\delta}}^2, \quad
		\bigl(\partial_\phi f_2(\phi), \phi \bigr)_{L^2} \geq C_6 \norm{\phi}_{L^4}^4 - C_7 \norm{\phi}_{L^2}^2. 
	\end{aligned}
	$$
\end{enumerate}
\end{assumption}
The growth condition and the semicoercivity regarding $\partial_c \phi$ follow directly from its affine linear structure, but we have still written it down for convenience.
Before addressing the free energy density in our applications of the Kim--Kim--Suzuki model, we assume a further property of $f$ that will be relevant for proving the uniqueness of solutions.
\begin{assumption}[Additional properties] \label{assump:free_energy_density2}
The local free energy density $f \in C^{1}(\R^2;\R)$ satisfies:\\[-0.5em]
\begin{enumerate}[label=\textup{(F\arabic*)}, ref=F\arabic*, start=5,leftmargin=.9cm] 
	\item \label{assump:growthenh} \textbf{Enhanced growth condition}: $\exists C_8>0$ such that, for any $\phi,c \in \R$,
	$$|\partial_\phi f(\phi,c)|\leq C_8(1+\abs{c}+\abs{\phi});$$
	\item \label{assump:lipenh} \textbf{Joint Lipschitz continuity}: $\exists C_9>0$ such that, for any $\phi,c \in \R$,
	\[\begin{aligned}
		|\partial_\phi f(\phi_1,c_1)-\partial_\phi f(\phi_2,c_2)|
		&\le C_9 \bigl( (1+\abs{c_1}+\abs{\phi_1})|\phi_1-\phi_2|+|c_1-c_2| \bigr).
	\end{aligned}\]
\end{enumerate}
\end{assumption}

\begin{example}\label{Ex:KKSdensity}
We consider the local free energy density from the Kim--Kim--Suzuki model \cite[Eq.~(15)]{kim1999phase}:
\begin{equation}\label{eq:KKSf}
	f(\phi,c) = A\bigl(c - h(\phi)(1 - c_L) - c_L\bigr)^2 + \omega g(\phi) \quad \forall \phi, c \in \mathbb{R},
\end{equation}
with constants \(A, \omega, c_L > 0\), and functions:\\[-0.5em]
\begin{itemize}
	\item $g$ is a double-well potential, e.g., $g(\phi) = \phi^2(1-\phi)^2$, which we may extend quadratically outside the interval $[0,1]$ to ensure the enhanced growth condition \eqref{assump:growthenh}.
	\item $h$ is of the form $h(\phi) =\phi^2(3-2\phi)$, which we extend quadratically outside of $[0,1]$ to obtain the Lipschitz continuity of $h$ as required for $h_1$ in \eqref{assump:affine}. 
\end{itemize} 
\ \\
We now verify that $f$ defined by \eqref{eq:KKSf} satisfies Assumptions \ref{assump:free_energy_density} and \ref{assump:free_energy_density2}. \\[-0.5em]
\begin{enumerate}[label=\textup{(F\arabic*)}, ref=F\arabic*]
	\item The splitting is a follows directly from \eqref{eq:KKSf}. Define $f_1(\phi, c)= A\bigl(c - h(\phi)(1 - c_L) - c_L\bigr)^2$ and
	$f_2(\phi) = \omega g(\phi)$.
	Clearly, $f_1 \in C^1(\R^2; \R)$ and $f_2 \in C^1(\R; \R_{\geq 0})$.
	
	\item The partial derivative with respect to $c$ is:
	\begin{equation} \label{Eq:DerivFc}
		\partial_c f(\phi,c) = \partial_c f_1(\phi,c) = 2A\bigl[c - h(\phi)(1 - c_L) - c_L\bigr] = 2A c + B h(\phi),
	\end{equation}
	where $B = -2A(1-c_L)$. This is of the required affine linear form, choosing the function $h_1=Bh$, which is Lipschitz continuous if $h$ is cut-off accordingly. The derivative with respect to $\phi$ is:
	\begin{equation} \label{Eq:DerivFphi}
		\partial_\phi f(\phi,c)=\partial_\phi f_1(\phi,c)+\omega g'(\phi) = -2A\bigl[c - h(\phi)(1 - c_L) - c_L\bigr] h'(\phi)(1-c_L) + \omega g'(\phi),
	\end{equation}
	which is again of the required form, choosing $h_2(\phi)=-2Ah'(\phi)(1-c_L)$ \ \\
	and $h_3(\phi)=2A[h(\phi)(1-c_L)+c_L]h'(\phi)(1-c_L)$.
	
	\item From \eqref{Eq:DerivFc}, we immediately get $|\partial_c f(\phi,c)| \lesssim 1 + \abs{c}$, satisfying the third inequality in \eqref{assump:growth}. 
	Since $h'$ is bounded, $h$ is Lipschitz continuous with $h(\phi)\lesssim \phi$ and $g'(\phi)=2\phi(1-\phi)(1-2\phi)$ has cubic growth (which is preserved under quadratic extension), we indeed obtain from \eqref{Eq:DerivFphi}
	\begin{align*}
		|\partial_\phi f_1(\phi,c)| &= |-2A\bigl[c - h(\phi)(1 - c_L) - c_L\bigr] h'(\phi)(1-c_L)| \lesssim 1 + \abs{c}+\abs{\phi}, \\
		|\partial_\phi f_2(\phi)| &= |\omega g'(\phi)| \lesssim 1 + \abs{\phi}^3.
	\end{align*}
	
	\item Using \eqref{Eq:DerivFc} and the definition of the $V_\delta$-seminorm, we obtain:
	\begin{align*}
		\bigl(\partial_c f(\phi,c), c \bigr)_{V_\delta} 
		&= A \iint_{\Omega\times\Omega} J_\delta(x-y) \bigl\{ [c(x)-c(y)] - (1-c_L)[h(\phi(x))-h(\phi(y))] \bigr\}  [c(x)-c(y)]  \dxy \\
		&= A \abs{c}_{V_\delta}^2 - A(1-c_L) \iint_{\Omega\times\Omega}  J_\delta(x-y) [h(\phi(x))-h(\phi(y))] (c(x)-c(y))  \dxy.
	\end{align*}
	Applying the Cauchy-Schwarz and Young's inequalities to the second term, and using the Lipschitz continuity of $h$ with $h(0)=0$ (which implies $|h(\phi)|_{V_\delta} \lesssim \abs{\phi}_{V_\delta}$), yields:
	\begin{align*}
		\bigl(\partial_c f(\phi,c), c \bigr)_{V_\delta} &\geq A \abs{c}_{V_\delta}^2 - \frac{A}{2} \abs{c}_{V_\delta}^2 - C |h(\phi)|_{V_\delta}^2 
		\geq \frac{A}{2} \abs{c}_{V_\delta}^2 - C \abs{\phi}_{V_\delta}^2.
	\end{align*}
	This is the first bound in \eqref{assump:semi}. For the second bound, we use the specific form of $g$:
	\begin{align*}
		\bigl(\partial_\phi f_2(\phi), \phi \bigr)_{L^2} &= \omega \int_\Omega g'(\phi) \phi  dx = \omega \int_\Omega 2(2\phi^4 - 3\phi^3 + \phi^2)  dx.
	\end{align*}
	Using Young's inequality ($-3\phi^3 \geq -(\phi^4 + (9/4)\phi^2)$) pointwise, we find:
	\begin{align*}
		2(2\phi^4 - 3\phi^3 + \phi^2) &\geq 2(2\phi^4 - \phi^4 - \frac{9}{4}\phi^2 + \phi^2) = 2\phi^4 - \frac{5}{2}\phi^2.
	\end{align*}
	Integrating gives  the second bound in \eqref{assump:semi}, that is, $\bigl(\partial_\phi f_2(\phi), \phi \bigr)_{L^2} \geq 2\omega \norm{\phi}_{L^4}^4 - C \norm{\phi}_{L^2}^2$.
	
	\item  From \eqref{Eq:DerivFphi}, the boundedness of $h$ and $h'$, and the fact that our quadratic extension of $g$ ensures $|g'(\phi)| \lesssim 1 + \abs{\phi}$, it indeed follows that:
	$$|\partial_\phi f(\phi,c)| \lesssim \vert c\vert + 1 + \abs{\phi}.$$
	
	\item  We analyze the difference using \eqref{Eq:DerivFphi}:
	\begin{align*}
		&|\partial_{\phi}f(\phi_1,c_1)-\partial_\phi f(\phi_2,c_2)| \lesssim |g'(\phi_1)-g'(\phi_2)| + |c_1 h'(\phi_1) - c_2 h'(\phi_2)| + |h(\phi_1)h'(\phi_1) - h(\phi_2)h'(\phi_2)|.
	\end{align*}
	The first term is handled by the Lipschitz continuity of $g'$ (due to its quadratic extension). For the second term, we obtain:
	\begin{align*}
		|c_1 h'(\phi_1) - c_2 h'(\phi_2)| &\leq |h'(\phi_2)| |c_1 - c_2| + \abs{c_1} |h'(\phi_1) - h'(\phi_2)| \\
		&\lesssim |c_1 - c_2| + \abs{c_1} |\phi_1 - \phi_2|.
	\end{align*}
	The third term is handled similarly using the Lipschitz continuity of $h$ and $h'$ (which follows due to the quadratic extension of $h$) 
	\begin{align*}
		|h(\phi_1) h'(\phi_1) - h(\phi_2) h'(\phi_2)| &\leq |h'(\phi_2)| |h(\phi_1) - h(\phi_2)| + |h(\phi_1)| |h'(\phi_1) - h'(\phi_2)| \\
		&\lesssim (1+ \abs{\phi_1}) |\phi_1 - \phi_2|.
	\end{align*}
	Combining these estimates finally yields the desired result, that is,
	$$
	|\partial_{\phi}f(\phi_1,c_1)-\partial_\phi f(\phi_2,c_2)| \lesssim (1 + \abs{c_1}+\abs{\phi_1}) |\phi_1 - \phi_2| + |c_1 - c_2|.
	$$
\end{enumerate}
\end{example}

\section{Local and nonlocal evolution equations}\label{sec:evolution}

In this section, the classical (local) KKS phase-field corrosion model as derived in \cite{kim1999phase} is reviewed. Next, its nonlocal extension by computing G\^{a}teaux derivatives of the corresponding nonlocal free energy and formulating nonlocal gradient flows are derived.

\subsection{Local KKS phase-field corrosion model}\label{subsec:local}

In the classical formulation of the KKS phase-field corrosion model, one defines a Ginzburg--Landau type local free energy functional by
\[
\mathcal{F}_{\mathrm{local}}[\phi,c] = \int_\Omega \Big\{ f(\phi,c) + \frac{\alpha_\phi}{2} |\nabla \phi|^2 + \frac{\alpha_c}{2} |\nabla c|^2 \Big\}\dx,
\]
where the free energy density $f \in C^1(\R^2;\R)$ is as specified in Assumption~\ref{assump:free_energy_density}. In this context, the parameters \(\alpha_\phi,\alpha_c>0\) quantify the energetic penalty associated with spatial variations in \(\phi\) and $c$, respectively, and thus control the thickness of the interfacial layer between the phases. The evolution of the phase-field $\phi$ is derived as the \(L^2\)- gradient flow of \(\mathcal{F}_{\mathrm{local}}\), and therefore the resulting equation is of Allen--Cahn type. However, the evolution of the concentration $c$ is obtained as an \(H^{-1}\)-gradient flow (which guarantees mass conservation), leading to a Cahn--Hilliard type equation. Specifically, one obtains the coupled Allen--Cahn/Cahn--Hilliard system
\begin{equation} \label{Eq:GradientFlowACCH}
	\begin{aligned}
		\partial_t \phi &= -L\frac{\delta \mathcal{F}_{\mathrm{local}}}{\delta \phi}, \qquad 
		\partial_t c = \nabla\cdot \left\{ M\nabla\frac{\delta \mathcal{F}_{\mathrm{local}}}{\delta c} \right\},
	\end{aligned}
\end{equation}
where \(L>0\) is the kinetic coefficient for the phase-field evolution and \(M>0\) is the mobility constant for the concentration evolution. A direct computation shows that
\[
\begin{aligned}
	\frac{\delta \mathcal{F}_{\mathrm{local}}}{\delta \phi} &= \partial_\phi f(\phi,c) - \alpha_\phi\Delta \phi, \qquad 
	\frac{\delta \mathcal{F}_{\mathrm{local}}}{\delta c} = \partial_c f(\phi,c)- \alpha_c\Delta c.
\end{aligned}
\]
Thus, the local KKS phase phield corrosion model takes the form
\begin{align}
	\partial_t \phi &= -L\Bigl\{\partial_\phi f(\phi,c) - \alpha_\phi\Delta\phi\Bigr\}, \label{eq:localAC}\\[1mm]
	\partial_t c &= \nabla\cdot\Bigl\{ M\nabla\Bigl(\partial_c f(\phi,c)- \alpha_c\Delta c\Bigr) \Bigr\}. \label{eq:localCH}
\end{align}
These equations have to be interpreted in a weak sense (with appropriate initial and boundary conditions) and, under standard assumptions, one can establish well-posedness (see, e.g., \cite{mai2016phase}). Note that the classical model does not incorporate an additional length scale beyond that determined by \(\alpha_\phi\), which motivates the nonlocal formulation.

\subsection{Nonlocal KKS phase-field corrosion model}
\label{subsec:nonlocal}

To incorporate a finite interaction horizon and introduce an intrinsic length scale, we replace local gradient terms by nonlocal integral operators. The nonlocal Helmholtz-type free energy, see e.g. \cite{fritz2019unsteady,gal2017nonlocal}, is defined by
\[
\begin{aligned}
	\mathcal{F}[\phi,c]
	&=
	\int_{\Omega} f\bigl(\phi(x),c(x)\bigr)\dx
	+
	\frac{\alpha_{\phi}}{4}
	\iint_{\Omega \times \Omega} J_{\delta}(x-y)[\phi(x)-\phi(y)]^2\dxy
	\\
	&+
	\frac{\alpha_{c}}{4}
	\iint_{\Omega \times \Omega} J_{\delta}(x-y)[c(x)-c(y)]^2\dxy,
\end{aligned}
\]
where \(J_{\delta}\) is the scaled kernel from Assumption~\ref{assump:kernel}, and \(\alpha_{\phi},a_c\geq 0\) are again penalty-like coefficients. Using the definition of the seminorm in $V_\delta$, see \eqref{assump:comp}, we may shortly write
\[
\mathcal{F}[\phi,c]
=
\int_{\Omega} f\bigl(\phi,c\bigr)\dx
+
\frac{\alpha_{\phi}}{4}
\abs{\phi}_{V_\delta}^2
+
\frac{\alpha_{c}}{4}
\abs{c}_{V_\delta}^2.
\]
The Gateaux derivatives (or: first variations) of \(\mathcal{F}\) at \((\phi,c)\) in directions \(\eta,\zeta\in C_0^\infty(\Omega)\) are defined by
\[
D_{\phi}\mathcal{F}[\phi,c]\{\eta\}
:=\left.\frac{d}{d\eps}\mathcal{F}[\phi+\eps\eta,c]\right|_{\eps=0},
\quad
D_{c}\mathcal{F}[\phi,c]\{\zeta\}
:=\left.\frac{d}{d\eps}\mathcal{F}[\phi,c+\eps\zeta]\right|_{\eps=0}.
\]
Moreover, the Gateaux derivatives relate to the functional derivatives $\frac{\delta\mathcal{F}}{\delta\phi}$ and $\frac{\delta\mathcal{F}}{\delta c}$ of $\mathcal{F}$ with repsect to $\phi$ and $c$, respectively, as follows
$$D_{\phi}\mathcal{F}[\phi,c]\{\eta\}=:\int_\Omega \frac{\delta\mathcal{F}}{\delta\phi} \eta  \dx, \qquad D_{c}\mathcal{F}[\phi,c]\{\zeta\}=:\int_\Omega \frac{\delta\mathcal{F}}{\delta c} \zeta  \dx.$$
First, the Gateaux derivative w.r.t.\ \(\phi\) reads
\[
\begin{aligned}
	D_{\phi}\mathcal{F}[\phi,c]\{\eta\}
	&=\int_{\Omega}\frac{d}{d\eps} f(\phi+\eps\eta,c)\Big|_{\eps=0}\dx \\
	&+\frac{\alpha_{\phi}}{4}
	\iint_{\Omega\times\Omega} J_{\delta}(x-y)\frac{d}{d\eps}\Bigl[(\phi+\eps\eta)(x)-(\phi+\eps\eta)(y)\Bigr]^2\Big|_{\eps=0}
	\dx\dy,
\end{aligned}
\]
and a short computation gives
\[
\begin{aligned} D_{\phi}\mathcal{F}[\phi,c]\{\eta\}
	&=\int_{\Omega}\partial_\phi f(\phi,c)\eta \dx
	+\frac{\alpha_{\phi}}{2}
	\iint_{\Omega\times\Omega}J_{\delta}(x-y)[\phi(x)-\phi(y)][\eta(x)-\eta(y)]\dx\dy \\
	&=(\partial_\phi f(\phi,c),\eta)_{L^2}
	+\alpha_{\phi} (\phi,\eta)_{V_\delta},
\end{aligned}
\]
which we may rewrite using the nonlocal operator $\mathcal{L}_\delta$, see \eqref{eq:V_nl_inner}, as 
\[
D_{\phi}\mathcal{F}[\phi,c]\{\eta\}
=\int_{\Omega}\bigl\{\partial_\phi f(\phi,c)
+\alpha_{\phi} \mathcal{L}_\delta \phi \bigr\} \eta\dx \quad \forall \eta \in C_0^\infty(\Omega).
\]
The fundamental lemma of the calculus of variations yields the functional derivative of $\mathcal{F}$ with respect to $\phi$, that is,
\begin{equation}
	\label{Eq:FuncDerivPhi}
	\dfrac{\delta\mathcal{F}}{\delta\phi}(x)
	=\partial_\phi f(\phi(x),c(x))
	+\alpha_{\phi}\mathcal{L}_\delta \phi.
\end{equation}
Using the same techniques, the Gateaux derivative w.r.t.\ \(c\) reads
\[
D_{c}\mathcal{F}[\phi,c]\{\zeta\}
=\int_{\Omega}\frac{d}{d\eps} f(\phi,c+\eps\zeta)\Big|_{\eps=0}\dx = \int_{\Omega} \bigl\{ \partial_c f(\phi,c) + \alpha_c \mathcal{L}_\delta c\bigr\} \zeta \dx.
\]
Invoking the fundamental lemma of the calculus of variations, it yields the following functional derivative of $\mathcal{F}$ with respect to $c$:
\begin{equation}
	\label{Eq:FuncDerivC}
	\dfrac{\delta\mathcal{F}}{\delta c}
	=\partial_c f(\phi,c) + \alpha_c \mathcal{L}_\delta c.
\end{equation}

We now formulate the nonlocal evolution equations as gradient flows of \(\mathcal{F}[\phi,c]\).  First, we consider the standard $L^2$-gradient flow for $\phi$ as in the local KKS model \eqref{Eq:GradientFlowACCH} but we replace the functional derivative of the local energy $\mathcal{F}_\text{local}$ by its nonlocal counterpart, that is \eqref{Eq:FuncDerivPhi}, which yields
\begin{equation}
	\label{Eq:NonlocalKKS}
	\begin{aligned}
		\partial_t \phi
		&= -L\frac{\delta \mathcal{F}}{\delta \phi}
		= -L\bigl\{\partial_\phi f(\phi,c)
		+ \alpha_{\phi}\mathcal{L}_\delta \phi \bigr\}. 
	\end{aligned}
\end{equation}
For \(\phi\), which is not mass conserved, the natural inner product is the \(L^2\)-scalar product.  For \(c\), mass conservation requires an \(H^{-1}\)-type structure. Thus, instead of assuming an $H^{-1}$-gradient flow governing $c$ as we have done in the local setting, we consider a nonlocal gradient flow based on the inner product in $V_\delta^*$. The dual space $V_\delta^*$ (that later becomes the nonlocal $H^{-1}$-type space) has the inner product:
\begin{equation}
	\langle u, v \rangle_{V_\delta^*}
	= \langle u, \mathcal{L}_\delta^{-1} v \rangle_{L^2}
	= \int_{\Omega} u(x) (\mathcal{L}_\delta^{-1} v)(x)  \dx.
	\label{eq:V_nl_dual_inner}
\end{equation}
Thus, the $V_\delta^*$ gradient flow reads 
$$
\partial_t c = -M \mathcal{L}_\delta \Big(\frac{\delta\mathcal{F}}{\delta c}\Big)=-M \mathcal{L}_\delta \big(\partial_c f(\phi,c)+\alpha_c \mathcal{L}_\delta c\big).
$$
Integration over \(\Omega\) and using $\int_\Omega \mathcal{L}_\delta u  \dx =0$ shows that the concentration equation conserves mass, that is, $\ddt \int_{\Omega} c\dx = 0$.
To prove energy dissipation, one computes
\[
\ddt\mathcal{F}[\phi,c]
= \int_{\Omega} \frac{\delta \mathcal{F}}{\delta \phi}\frac{\partial \phi}{\partial t}\dx
+ \int_{\Omega} \frac{\delta \mathcal{F}}{\delta c}\frac{\partial c}{\partial t}\dx.
\]
Substituting \(\partial_t \phi = -L\delta\mathcal{F}/\delta\phi\) and \(\partial_t c = M\mathcal{L}_\delta(-\delta\mathcal{F}/\delta c)\) leads to
\[
\ddt\mathcal{F}[\phi,c]
= -L \int_{\Omega} \Bigl|\frac{\delta \mathcal{F}}{\delta \phi}\Bigr|^2\dx
- M \int_{\Omega} \frac{\delta \mathcal{F}}{\delta c}\mathcal{L}_\delta\Bigl(\frac{\delta \mathcal{F}}{\delta c}\Bigr)\dx,
\]
from which we conclude dissipativity \(\ddt\mathcal{F}\le 0\) as \(J_\delta\) is symmetric, see Assumption~\ref{assump:free_energy_density}, and hence
\[
\int_{\Omega} \frac{\delta \mathcal{F}}{\delta c}\mathcal{L}_\delta\Bigl(\frac{\delta \mathcal{F}}{\delta c}\Bigr)\dx
= \frac{1}{2}\iint_{\Omega\times\Omega} J_{\delta}(x-y)\Bigl[\frac{\delta \mathcal{F}}{\delta c}(x)
- \frac{\delta \mathcal{F}}{\delta c}(y)\Bigr]^2\dx\dy
\ge0.
\]
In total, we consider in this work the nonlocal KKS phase-field corrosion model
\begin{align} \label{Eq:ModelNonlocalAC}
	\partial_t \phi
	&= -L\bigl[\partial_\phi f(\phi,c)
	+ \alpha_{\phi}\mathcal{L}_\delta \phi \bigr], \\ \label{Eq:ModelNonlocalCH}\partial_t c &= -M \mathcal{L}_\delta[\partial_c f(\phi,c)+\alpha_c \mathcal{L}_\delta c].
\end{align}
Intuitively, one could say that we simply replaced all Laplace operators $-\Delta$ in the local KKS model \eqref{eq:localAC}--\eqref{eq:localCH} by the nonlocal operator $\mathcal{L}_\delta$. And indeed,
in the limit \(\delta\to0\) with the properly normalized kernel (see Section~\ref{sec:limit}), the nonlocal operators converge to the Laplacian and local diffusion, recovering the classical model given by \eqref{eq:localAC}-\eqref{eq:localCH}.

\section{Well-posedness analysis} 
\label{sec:wellposed_analysis}
In this section, we establish well-posedness for the nonlocal KKS phase-field corrosion model \eqref{Eq:ModelNonlocalAC}--\eqref{Eq:ModelNonlocalCH}.  Since in the fully coupled system the nonlocal Allen--Cahn equation \eqref{Eq:ModelNonlocalAC} depends on \(c\) and the Cahn--Hilliard equation \eqref{Eq:ModelNonlocalCH} depends on \(\phi\), we ultimately obtain the well-posedness of a nonlocal coupled Allen--Cahn/Cahn--Hilliard system. We prove the existence, uniqueness, and continuous dependence of weak solutions by constructing Galerkin approximations, deriving uniform energy estimates, and applying compactness arguments to pass to the limits. 

\begin{assumption} \label{assumption}
	We assume:\\[-0.5em]
	\begin{enumerate}[label=(A\arabic*), ref=A\arabic*, leftmargin=.9cm,itemsep=0em]
		\item $\Omega \in \R^d$, $d \in \{1,2,3\}$, bounded Lipschitz domain, $T>0$ is a fixed time;
		\item $c_0,\phi_0 \in L^2$;
		\item $\alpha_\phi,L,M>0$;
		\item $f$ fulfills Assumption~\ref{assump:free_energy_density};
		\item $J_\delta$, $\delta>0$, fulfills Assumption~\ref{assump:kernel};
		\item either $\alpha_c=0$ or $\alpha_c \geq 0$ $\wedge$ ($J_\delta \in L^1$ $\vee$ $c_0 \in V_\delta$).
	\end{enumerate}
\end{assumption}
As we will see in the proof, the value of $\alpha_c$ is delicate. One option is assuming $\alpha_c=0$ which means that the nonlocal effects solely appear due to the phase-field and then influence the concentration through the coupling mechanism. In the case of $\alpha_c>0$, we either require $J_\delta \in L^1$ (which is true for the top-hat but not for the truncated fractional kernel) or $c_0 \in V_\delta$. We begin by stating the definition of a weak solution to \eqref{Eq:ModelNonlocalAC}--\eqref{Eq:ModelNonlocalCH}. 

\begin{definition}\label{Def:Sol}
	A tuple $(\phi,c)$ with the regularity
	\[
	\begin{aligned}
		\phi&\in L^\infty(0,T;L^2(\Omega))
		\cap L^2(0,T;V_\delta),
		\qquad
		\partial_t\phi\in L^2(0,T;V_\delta^{*}), \\
		c&\in L^\infty(0,T;L^2(\Omega))
		\cap L^2(0,T;V_\delta),
		\qquad
		\partial_tc\in L^2(0,T;V_\delta^{*}),
	\end{aligned}
	\]
	is a \emph{weak solution} to \eqref{Eq:ModelNonlocalAC}--\eqref{Eq:ModelNonlocalCH} if for a.e.\ $t\in (0,T)$
	\begin{equation}\label{weakform}
		\begin{aligned}
			\langle\partial_t\phi(t),v\rangle
			&=-L \bigl(\partial_\phi f(\phi(t),c(t)),v\bigr)_{L^2}
			-\alpha_\phi L \bigl(\phi(t),v\bigr)_{V_\delta}
			&&\quad\forall v\in V_\delta,
			\qquad
			\phi(0)=\phi_0, \\
			\langle\partial_tc(t),v\rangle
			&=-M \bigl(\partial_c f(\phi(t),c(t)),v\bigr)_{V_\delta}-\alpha_c M\bigl(\mathcal{L}_\delta c(t),\mathcal{L}_\delta v\bigr)_{L^2}
			&&\quad\forall v\in V_\delta,
			\qquad
			c(0)=c_0.
		\end{aligned}
	\end{equation}
\end{definition}
\ \\
We note that the inner product of the term with $\alpha_c$ appears using the self-adjointness of $\mathcal{L}_\delta$. 
\begin{theorem}
	\label{thm:NLAC}
	Let Assumption~\ref{assumption} hold. Then
	there exists a weak solution
	$(\phi,c)$ of \eqref{Eq:ModelNonlocalAC}--\eqref{Eq:ModelNonlocalCH} in the sense of Definition~\ref{Def:Sol} which further satisfies
	the energy inequality
	\begin{equation} \label{EnergyIneq} \|\phi(t)\|_{L^2}^2+\|c(t)\|_{L^2}^2 + \|\phi\|_{L^2(0,T; V_\delta)}^2+\|c\|_{L^2(0,T;V_\delta)}^2 \leq C(T) \big( \|\phi_0\|_{L^2}^2+\|c_0\|_{L^2}^2 \big) \quad \forall t>0.
	\end{equation}
	In addition, if $J_\delta$ fulfills Assumption~\ref{assump:free_energy_density2} and $(\phi_1,c_1)$ is a strong solution in the sense that $\phi_1,c_1\in L^1(0,T;L^\infty)$ and $(\phi_2,c_2)$ is a weak solution, then the solutions depend continuously on the data, that is, 
	\[
	\|\phi_1(t)-\phi_2(t)\|_{L^{2}}+\|c_1(t)-c_2(t)\|_{L^{2}}
	\lesssim  \|\phi_{1}(0)-\phi_{2}(0)\|_{L^{2}}+\|c_{1}(0)-c_{2}(0)\|_{L^{2}}.
	\]
	Especially, if the initial data is the same for both solutions, then it holds $(\phi_1,c_1)=(\phi_2,c_2)$, which implies weak-strong uniqueness.
\end{theorem}

\begin{proof} \emph{Step 1 (Galerkin approximation).}
	Since $\mathcal L_\delta^{-1}:\mathring{L}^2 \to \mathring{L}^2$ is self-adjoint and compact (see \eqref{assump:comp} in Assumption~\ref{assump:kernel}), we invoke the spectral theorem to obtain the existence of an orthonormal basis $(v_k)_{k\ge1}$ in $\mathring{L}^2$, which are further orthogonal in $V_\delta$. For $N\in\N$, we set $V_N:=\operatorname{span}\{v_1,\dots,v_N\}$ and
	$$\phi_{N}(t):=\sum_{i=1}^{N}a_i^{(N)}(t)v_i, \quad c_{N}(t):=\sum_{i=1}^{N}b_i^{(N)}(t)v_i$$ with time-dependent coefficient functions
	$a^{(N)}_i,b^{(N)}_i\in C^{1}$ that shall satisfy the Galerkin system (for any $j \in \{1,\dots,N\}$ and a.e.~$t \in (0,T)$)
	\begin{equation}\label{weakform-galerkin}
		\begin{aligned}
			(\partial_t\phi_N(t),v)_{L^2}
			&=-L \bigl(\partial_\phi f(\phi_N(t),c_N(t)),v\bigr)_{L^2}
			-\alpha_\phi  L \bigl(\phi_N(t),v\bigr)_{V_\delta}
			&&~\forall v\in V_N,
			\quad
			\phi_N(0)=\Pi_{V_N} \phi_0, \\
			(\partial_tc_N(t),v)_{L^2}
			&=-M \bigl(\partial_c f(\phi_N(t),c_N(t)),v\bigr)_{V_\delta} -\alpha_c M\bigl(\mathcal{L}_\delta c_N(t),\mathcal{L}_\delta v\bigr)_{L^2}
			&&~\forall v\in V_N,
			\quad
			c_N(0)=\Pi_{V_N}c_0.
		\end{aligned}
	\end{equation}
	We may rewrite the Galerkin system via the coefficient functions, using the orthonormality of the basis functions, which gives
	\begin{equation}\label{galerkin-ODE2}
		\begin{aligned}
			\ddt a^{(N)}_j(t)
			&= -L\bigl(\partial_\phi f(\phi_N,c_N),v_j\bigr)_{L^2}
			-\alpha_\phi L\bigl(\phi_N,v_j\bigr)_{V_\delta},
			&&\quad 
			a^{(N)}_j(0)=(\phi_0,v_j)_{L^2}, \\
			\ddt b^{(N)}_j(t)
			&= -M\bigl(\partial_c f(\phi_N,c_N),v_j\bigr)_{V_\delta} -\alpha_c M\bigl(\mathcal{L}_\delta c_N,\mathcal{L}_\delta v_j\bigr)_{L^2},
			&&\quad 
			b^{(N)}_j(0)=(c_0,v_j)_{L^2}.
		\end{aligned}
	\end{equation}
	As the involved nonlinearities are continuous in the arguments, we invoke the Cauchy--Peano theorem that gives a continuous solution up to some maximal horizon $T_N\leq T$, that is, $(\phi_N,c_N) \in C([0,T_N];V_N\times V_N)$. \smallskip
	
	\noindent\emph{Step 2 (Energy estimates).}
	Testing \eqref{weakform-galerkin}$_1$ with $v=\phi_N(t)$ yields
	\begin{equation}\label{energy-id}
		\frac12 \ddt \norm{\phi_N}_{L^2}^2
		+\alpha_\phi L \abs{\phi_N}_{V_\delta}^2
		+L(\partial_\phi f(\phi_N,c_N),\phi_N)_{L^2} =0.
	\end{equation}
	Using the splitting $f=f_1+f_2$, see \eqref{assump:splitting}, and the growth conditions \eqref{assump:growth} and semicoercivity \eqref{assump:semi} of $f$, we have with Young's and Hölder's inequality that
	$$\begin{aligned} 
		\bigl|\bigl(\partial_\phi f_1(\phi_N,c_N), \phi_N \bigr)_{L^2}\bigr| &\leq C_1 (1+\abs{c_N}+\abs{\phi_N},\abs{\phi_N})_{L^2} \lesssim 1+\norm{c_N}_{L^2}^2 + \norm{\phi_N}_{L^2}^2, \\
		\bigl(\partial_\phi f_2(\phi_N), \phi_N \bigr)_{L^2} &\geq C_6\norm{\phi_N}_{L^4}^4 - C_7 \norm{\phi_N}_{L^2}^2.
	\end{aligned}$$ 
	We insert this bound back into \eqref{energy-id} to obtain the following:
	\begin{equation} \label{Energy:Est1}
		\frac12 \ddt \norm{\phi_N}_{L^2}^2 + \alpha_\phi L  \abs{\phi_N}_{V_\delta}^2+ LC_6  \norm{\phi_N}_{L^4}^4
		\lesssim 1+ \norm{\phi_N}_{L^2}^2 +  \norm{c_N}_{L^2}^2.
	\end{equation}
	
	On the right-hand side of the inequality is the $L^2$-norm of $c_N$ that we still have to bound. Indeed, testing \eqref{weakform-galerkin}$_2$ with $v=Kc_N(t)$, where $K>0$ will be determined later, it yields the estimate
	$$\frac{K}{2} \ddt\norm{c_N}_{L^2}^2 + \alpha_c KM \norm{\mathcal{L}_\delta c_N}_{L^2}^2   + M K (\partial_c f(\phi_N,c_N),c_N)_{V_\delta}=0.$$
	We make use of the semicoercivity \eqref{assump:semi} as assumed in Assumption~\ref{assump:free_energy_density}, to conclude
	$$\bigl(\partial_c f(\phi_N,c_N), c_N \bigr)_{V_\delta} \geq C_4 \abs{c_N}_{{V_\delta}}^{2} - C_5 \abs{\phi_N}_{{V_\delta}}^2,$$
	and we insert this back into the estimate to deduce
	$$\frac{K}{2} \ddt \norm{c_N}_{L^2}^2 + \alpha_c KM \norm{\mathcal{L}_\delta c_N}_{L^2}^2+ C_4 M K \abs{c_N}_{V_\delta}^2 \leq C_5 M K  \abs{\phi_N}_{{V_\delta}}^2.$$
	We add this estimate to \eqref{Energy:Est1}, which gives in total
	$$\begin{aligned} &\frac12 \ddt \norm{\phi_N}_{L^2}^2 + \alpha_\phi L  \abs{\phi_N}_{V_\delta}^2+ LC_6  \norm{\phi_N}_{L^4}^4+\frac{K}{2} \ddt \norm{c_N}_{L^2}^2  +\alpha_c KM \norm{\mathcal{L}_\delta c_N}_{L^2}^2+ C_4 MK \abs{c_N}_{V_\delta}^2
		\\ &\le C+ C \norm{\phi_N}_{L^2}^2 + C \norm{c_N}_{L^2}^2 + MK C_5 \abs{\phi_N}_{{V_\delta}}^2,
	\end{aligned}$$
	where the generic constant $C$ is independent of $K$.
	At this point, we select $K>0$ sufficiently small to ensure the positivity of $\alpha_\phi L - M K C_5 > 0$.
	Defining $y_N(t)=1+\norm{\phi_N}_{L^2}^2 + K\norm{c_N}_{L^2}^2$ yields 
	\[
	\ddt y_N(t) \leq C y_N(t),
	\]
	and Gronwall's inequality \cite[Lemma II.4.10]{boyer2012mathematical} finally yields
	\begin{equation}\label{L2-Gr}
		\norm{\phi_N(t)}_{L^{2}}^{2}+K\norm{c_N(t)}_{L^{2}}^{2}
		\lesssim \bigl(1+\|\phi_N(0)\|_{L^{2}}^{2}+K\|c_N(0)\|_{L^{2}}^{2}\bigr)
		e^{Ct},
		\qquad
		0\le t<t_N^{\max}.
	\end{equation}
	At this point, we use that the initial data is chosen as the orthogonal projection onto $V_N$, which gives $$\|\phi_N(0)\|_{L^2}=\|\Pi_{V_N} \phi_0 \|_{L^2} \leq \norm{\phi_0}_{L^2},$$
	and in the same manner for $c_N(0)$.
	Hence, bound \eqref{L2-Gr} becomes independent of $N$, thus we may argue by the finite time blow-up theorem for ordinary differential
	equations to conclude $t_N^{\max}\ge T$ and thus the
	Galerkin solution is global on $[0,T]$. Next, we integrate \eqref{Energy:Est1} over $(0,T)$
	and use the estimate \eqref{L2-Gr}, which yields the $N$-uniform bound:
	\begin{equation}
		\label{X-bound}   \int_{0}^{T}\norm{\phi_N(t)}_{L^4}^{4}+\abs{\phi_N(t)}_{V_\delta}^{2}+\abs{c_N(t)}_{V_\delta}^{2} \dt
		\lesssim 1+\norm{\phi_0}_{L^{2}}^{2}+\norm{c_0}_{L^{2}}^{2}.
	\end{equation}
	%
	
	
	\noindent\emph{Step 3 (Optional bound).} We require an additional bound in the case of $\alpha_c>0$ and $J_\delta \notin L^1$ (e.g. the truncated fractional kernel) as we will see in the later steps. In this case, we further have $\mathcal{L}_\delta c_0 \in L^2$ according to Assumption~\ref{assumption}. We note that $\mathcal{L}_\delta^2 c_N$ is a valid test function in \eqref{weakform-galerkin}$_2$ due to the eigenbasis choice, and thus, we obtain
	$$\frac12 \ddt \norm{\mathcal{L}_\delta c_N}_{L^2}^2 + \alpha_c M \norm{\mathcal{L}_\delta^2 c_N}_{L^2}^2 + M(\partial_c f(\phi_N,c_N),\mathcal{L}_\delta^2 c_N)_{L^2}=0.$$
	We apply the growth estimate \eqref{assump:growth} on $\partial_c f(\phi_N,c_N)$ and integrate the inequality in time to obtain with Young's inequality
	$$\frac12 \norm{\mathcal{L}_\delta c_N(t)}_{L^2}^2 + \alpha_c M \norm{\mathcal{L}_\delta^2 c_N}_{L^2(0,t;L^2)}^2 \lesssim 1+\norm{\mathcal{L}_\delta c_0}_{L^2}^2+\norm{c_N}_{L^2(0,t;L^2)}^2+\norm{\phi_N}_{L^2(0,t;L^2)}^2.$$
	The right-hand side is uniformly bounded to the previous energy estimate, which yields 
	\begin{equation}\label{Eq:Optional} \norm{\mathcal{L}_\delta c_N}_{L^\infty(L^2)}^2 +  \norm{\mathcal{L}_\delta^2 c_N}_{L^2(L^2)}^2 \leq C.
	\end{equation}
	
	\noindent\emph{Step 4 (Time derivative bounds).} Lastly, we derive a bound of the time derivative of $\phi_N$ and we proceed as in \cite[Lemma V.1.6]{boyer2012mathematical}. We consider an arbitrary function $\xi \in L^2(0,T;V_\delta)$ and its orthogonal projection is given by $\Pi_{V_N} \xi := \sum_{i=1}^N (\xi,w_i)_{V_\delta} w_i\in V_N$ for a.e.~$t \in (0,T)$. We test \eqref{galerkin-ODE2}$_1$ with $\Pi_{V_N} \xi$, take the sum over $j=1,\dots,N$, take the absolute value and integrate in time to deduce by the Hölder inequality that
	$$\begin{aligned}
		\int_0^T \abs{\langle \partial_t \phi_N, \Pi_{V_N} \xi \rangle} \dt  &= \int_0^T \Big|-\alpha_\phi L (\phi_N,\Pi_{V_N} \xi)_{V_\delta} - L (\partial_\phi f(\phi_N,c),\Pi_{V_N} \xi)_{L^2}\Big| \dt  \\
		&\lesssim  \norm{\phi_N}_{L^2(V_\delta)} \|\Pi_{V_N} \xi\|_{L^2(V_\delta)} +  (1+\abs{c_N}+\abs{\phi_N}^3,|\Pi_{V_N} \xi|)_{L^2(L^2)} \\
		&\lesssim \norm{\phi_N}_{L^2(V_\delta)} \|\Pi_{V_N} \xi\|_{L^2(V_\delta)} + \bigl(1+\norm{c_N}_{L^2(L^2)}\bigr) \|\Pi_{V_N} \xi\|_{L^2(L^2)} +   \norm{\phi_N}_{L^4(L^4)}^3 \|\Pi_{V_N} \xi\|_{L^4(L^4)} 
	\end{aligned}$$
	At this point, we use on the left-hand side that the orthogonal projection is self-adjoint to deduce 
	$$\langle \partial_t \phi_N, \Pi_{V_N} \xi \rangle = \langle \partial_t \Pi_{V_N}^* \phi_N,  \xi \rangle= \langle \partial_t  \phi_N,  \xi \rangle \quad \forall \xi \in L^2(0,T;L^2),$$
	and on the right-hand side that $\|\Pi_{V_N} \xi\|_{V_\delta} \leq \|\xi\|_{V_\delta}$, apply the previous uniform bounds \eqref{L2-Gr}--\eqref{X-bound} and take the supremum over all $\xi \in L^4(0,T;V_\delta)$ with $\|\xi\|_{L^4(V_\delta)} \leq 1$ to obtain the uniform bound
	\begin{equation} \label{Eq:Derivativebound}
		\begin{aligned} \|\partial_t \phi_N\|_{L^{4/3}(V_\delta^*)} &\leq C.
		\end{aligned}
	\end{equation}
	We note that attaining a similar estimate for $\partial_t c_N$ is not as straight-forward due to the term with $\alpha_c$ which is not bounded by the $V_\delta$-norm in the case of the truncated fractional kernel. In the case of $J_\delta \in L^1$ (as for the top-hat kernel), we conclude
	\[
	\abs{\mathcal{L}_\delta v(x)} \leq \int_{\Omega} J_\delta(x-y) \abs{v(x)-v(y)}  dy \leq \left( \int_{\Omega} J_\delta(x-y)  dy \right)^{1/2} \left( \int_{\Omega} J_\delta(x-y) \abs{v(x)-v(y)}^2  dy \right)^{1/2},
	\]
	and it follows $\|\mathcal{L}_\delta \xi\|_{L^2} \lesssim \|\xi\|_{V_\delta}$ for any $\xi \in V_\delta$.
	Using this inequality, we proceed as before to deduce
	$$\begin{aligned}
		\int_0^T \abs{\langle \partial_t c_N, \Pi_{V_N} \xi \rangle} dt &= \int_0^T \Big|- \alpha_c M(\mathcal L_\delta c_N, \mathcal L_\delta \Pi_{V_N} \xi)_{L^2} - M (\partial_c f(\phi_N,c_N),\Pi_{V_N} \xi)_{V_\delta} \Big| \dt  \\
		&\lesssim \norm{\mathcal L_\delta c_N}_{L^2(L^2)}   \norm{\mathcal L_\delta \Pi_{V_N} \xi}_{L^2(L^2)}+(1+\abs{\phi_N}+\abs{c_N},\abs{\mathcal L_\delta\Pi_{V_N} \xi})_{L^2} \\
		&\lesssim  \big(1+\norm{\mathcal L_\delta c_N}_{L^2(L^2)}+\norm{\phi_N}_{L^2(L^2)}+\norm{c_N}_{L^2(L^2)}\big) \| \xi\|_{L^2(V_\delta)}. 
	\end{aligned}$$
	When $J_\delta \notin L^1$, then we can use the optional bound \eqref{Eq:Optional} to conclude
	$$\begin{aligned}
		\int_0^T \abs{\langle \partial_t c_N, \Pi_{V_N} \xi \rangle} dt 
		&\lesssim \norm{\mathcal L_\delta^2 c_N}_{L^2(L^2)}   \norm{\Pi_{V_N} \xi}_{L^2(L^2)}+(1+\abs{\phi_N}+\abs{c_N},\abs{\mathcal L_\delta\Pi_{V_N} \xi})_{L^2} \\
		&\lesssim  \big(1+\norm{\mathcal L_\delta^2 c_N}_{L^2(L^2)}+\norm{\phi_N}_{L^2(L^2)}+\norm{c_N}_{L^2(L^2)}\big) \| \xi\|_{L^2(W_\delta)},
	\end{aligned}$$
	where $W_\delta \subset V_\delta$ is the space of $L^2$-functions $w$ with $\|\mathcal L_\delta w\|_{L^2} < \infty$. In the case of $J_\delta \in L^1$ it holds $W_\delta=V_\delta$. In summary, we obtain
	\begin{equation} \label{Eq:Derivativebound2}\norm{\partial_t c_N}_{L^2(W_\delta^*)} \leq C.\end{equation}
	
	\noindent\emph{Step 5 (Passing to the limits).} By the uniform bounds \eqref{L2-Gr}--\eqref{Eq:Derivativebound}, the Eberlein--\v{S}mulian theorem \cite[Theorem 3.19]{brezis2011functional} guarantees subsequences (not relabelled) and  limit functions $\phi$ and $c$ with the weak/weak-$*$ convergences
	\[
	\begin{alignedat}{3}
		\phi_N&\rightharpoonup^*\phi\text{ in }
		L^{\infty}(0,T;L^2), \quad &&\phi_N\rightharpoonup\phi\text{ in }
		L^{2}(0,T;V_\delta) \cap L^4(0,T;L^4),
		\quad
		&&\partial_t\phi_N\rightharpoonup\partial_t\phi
		\text{ in } L^{4/3}(0,T;V_\delta^*), \\
		c_N&\rightharpoonup^* c\text{ in }
		L^{\infty}(0,T;L^2), \quad &&c_N\rightharpoonup  c\text{ in }
		L^{2}(0,T;V_\delta),
		\quad
		&&\partial_t c_N\rightharpoonup\partial_t c
		\text{ in } L^{2}(0,T;W_\delta^*).  
	\end{alignedat}
	\]
	Furthermore, we have $\alpha_c\mathcal L_\delta c_N\rightharpoonup\alpha_c\mathcal L_\delta c$ in $L^2(0,T;L^2)$ but we stress that we need this convergence result only for treating the convergence of the term with $\alpha_c$ in the variational form and nowhere else, that is, $\alpha_c \geq 0$ (and possibly $\alpha_c=0$) is a valid assumption.
	Moreover, considering the Gelfand triple $V_\delta \Subset L^2 \cong (L^2)^* \subset V_\delta^*$, see \eqref{assump:comp} in Assumption~\ref{assump:kernel}, the Aubin-Lions lemma \cite[Theorem II.5.16]{boyer2012mathematical} states the compact embeddings
	$$L^2(0,T;V_\delta) \cap L^{4/3}(0,T;V_\delta^*) \Subset L^2(0,T;L^2), \quad L^\infty(0,T;L^2) \cap L^{4/3}(0,T;V_\delta^*) \Subset C([0,T];V_\delta^*),$$
	which yield the strong convergences
	\begin{equation} \label{Eq:StrongConv}
		\phi_N\to\phi \text{ in }L^{2}(0,T;L^{2}),
		\quad
		\phi_N\to\phi \text{ in }C([0,T];V_\delta^*).
	\end{equation}
	Similarly, using the Gelfand triple $V_\delta \Subset L^2 \cong (L^2)^* \subset V_\delta^*\subset W_\delta^*$ gives 
	the strong convergences
	\[
	c_N\to c \text{ in }L^{2}(0,T;L^{2}),
	\quad
	c_N\to c \text{ in }C([0,T];V_\delta^*).
	\]
	We multiply the Galerkin system \eqref{weakform-galerkin} by an arbitrary function $\eta \in C_c^\infty(0,T)$, giving
	\begin{equation}\label{weakform-galerkintime}
		\begin{aligned}
			\int_0^T \langle \partial_t\phi_N(t),v\rangle \eta(t) \dt
			+L \int_0^T \left[ \bigl(\partial_\phi f(\phi_N(t),c_N(t)),v\bigr)_{L^2}
			+\alpha_\phi   \bigl(\phi_N(t),v\bigr)_{V_\delta} \right] \eta(t) \dt&=0
			&&~\forall v\in V_N, \\
			\int_0^T \langle \partial_tc_N(t),v\rangle \eta(t) \dt
			+M \int_0^T \left[ \bigl(\partial_c f(\phi_N(t),c_N(t)),v\bigr)_{V_\delta} +\alpha_c \bigl(\mathcal{L}_\delta c_N(t),\mathcal{L}_\delta v\bigr)_{L^2} \right] \eta(t) \dt&=0
			&&~\forall v\in V_N,
		\end{aligned}
	\end{equation}
	Regarding the passage of the limit $N \to \infty$, we can easily pass to the limits in the linear terms using the weak/weak-$*$ convergence results. We solely note that $\alpha_c \bigl(\mathcal{L}_\delta c_N(t),\mathcal{L}_\delta v\bigr)_{L^2}$ converges for $\alpha_c>0$ by using the weak convergence of $\mathcal{L}_\delta c_N$ in $L^2(0,T;L^2)$ and $J_\delta \in L^1$, see \eqref{assumption}, which gives $\norm{\mathcal{L}_\delta v}_{L^2} \lesssim |v|_{V_\delta}$. If $J_\delta \notin L^1$, then we rewrite the term as $\alpha_c \bigl(\mathcal{L}_\delta^2 c_N(t),v\bigr)_{L^2}$ and use the weak convergence of $\mathcal{L}_\delta^2 c_N$ in $L^2(0,T;L^2)$ from the optional estimate \eqref{Eq:Optional}.  Thus, it remains to investigate the terms involving the nonlinear functions $f$.
	\begin{itemize}
		\item We use that $\partial_c f(\phi_N,c_N)$ is of the form $Ac+h_1(\phi)$ with $A>0$ and $h_1 \in C^{1,1}(\R)$, see \eqref{assump:affine}, and that $c_N$ already converges weakly in $L^2(0,T;V_\delta)$. Regarding the convergence involving the nonlinear bounded function $h_1$, we argue by the Lebesgue dominated convergence theorem \cite[Theorem 4.2]{brezis2011functional}. First, we recall the definition of the semi-inner product of $V_\delta$
		\[
		\bigl(h_1(\phi_N), v\bigr)_{V_\delta} = \frac{1}{2} \iint_{\Omega \times \Omega} J_\delta(x-y) \bigl[h_1(\phi_N(x)) - h_1(\phi_N(y))\bigr] [v(x) - v(y)]  \dxy.
		\]
		We define
		\[
		F_N(t, x, y) = J_\delta(x-y) \Bigl\{ \bigl[h_1(\phi_N(t,x)) - h_1(\phi_N(t,y))\bigr] - \bigl[h_1(\phi(t,x)) - h_1(\phi(t,y))\bigr] \Bigr\} [v(x) - v(y)] \eta(t).
		\]
		We want to show that
		\[
		\iiint_{[0,T] \times \Omega \times \Omega} F_N(t, x, y)  \,\textup{d}(t,x,y) \to 0.
		\]
		Since \(\phi_N \to \phi\) almost everywhere in \((0,T) \times \Omega\) and \(h_1\) is continuous, we have
		$h_1(\phi_N(x)) - h_1(\phi_N(y)) \to h_1(\phi(x)) - h_1(\phi(y))$ almost everywhere.
		Hence, \(F_N(t, x, y) \to 0\) almost everywhere.
		Since \(h_1\) is Lipschitz continuous, see \eqref{assump:affine}, it holds
		\[
		\left| \bigl[h_1(\phi_N(x)) - h_1(\phi_N(y))\bigr] - \bigl[h_1(\phi(x)) - h_1(\phi(y))\bigr] \right| \lesssim \abs{\phi_N(x) - \phi_N(y)} + \abs{\phi(x) - \phi(y)}.
		\]
		Therefore,
		\[
		\abs{F_N(t, x, y)} \lesssim J_\delta(x-y) \left( \abs{\phi_N(x) - \phi_N(y)} + \abs{\phi(x) - \phi(y)} \right) \abs{v(x) - v(y)} \abs{\eta(t)}.
		\]
		Now, define the dominating function:
		\[
		G(t, x, y) = C J_\delta(x-y) \left( \abs{\phi_N(x) - \phi_N(y)} + \abs{\phi(x) - \phi(y)} \right) \abs{v(x) - v(y)} \abs{\eta(t)}.
		\]
		We can easily bound the integral of $G$ as follows
		\[
		\iiint_{[0,T] \times \Omega \times \Omega}  G(t, x, y)  \,\textup{d}(t,x,y) \lesssim \|\eta\|_{L^\infty} \int_0^T \left( \abs{\phi_N(t)}_{V_\delta} + \abs{\phi(t)}_{V_\delta} \right) \abs{v}_{V_\delta}  \dt,
		\]
		which implies \(G \in L^1([0,T] \times \Omega \times \Omega)\).
		Since \(|F_N(t, x, y)| \leq G(t, x, y)\) and \(F_N \to 0\) almost everywhere, by the Lebesgue dominated convergence theorem, we conclude:
		\[
		\iiint_{[0,T] \times \Omega \times \Omega} F_N(t, x, y)  \,\textup{d}(t,x,y) \to 0.
		\]
		This implies
		\[
		\int_0^T \bigl(h_1(\phi_N(t)), v\bigr)_{V_\delta} \eta(t)  \dt \to \int_0^T \bigl(h_1(\phi(t)), v\bigr)_{V_\delta} \eta(t)  \dt,
		\]
		and finally
		$$\partial_c f(\phi_N,c_N) \rightharpoonup \partial_c f(\phi,c) \text{ in } L^2(0,T;V_\delta),$$
		
		\item Regarding the convergence of $\partial_\phi f(\phi_N,c_N)$, we argue in a similar manner. According to \eqref{assump:splitting} and \eqref{assump:affine}, it is of the form
		$\partial_\phi f(\phi_N, c_N) = c_N h_2(\phi_N) + h_3(\phi_N) + \partial_\phi f_2(\phi_N)$ and these functions satisfy
		$|h_2(\phi_N)| \leq C$, $|h_3(\phi_N)| \leq C(1 + \abs{\phi_N})$ and $|\partial_\phi f_2(\phi_N)| \lesssim 1+ \abs{\phi_N}^3$.
		It holds $c_N \rightharpoonup c$ in $L^2(0,T; L^2)$ and $h_2(\phi_N) \to h_2(\phi)$ in $L^p(0,T; L^p)$ for any $p < \infty$ by the Lebesgue dominated convergence theorem. 
		Since $h_2(\phi_N) v \eta \to h_2(\phi) v\eta $ strongly in $L^2(0,T; L^2)$ and $c_N \rightharpoonup c$ in $L^2(0,T; L^2)$, it follows that by the weak-strong convergence lemma 
		$$\int_0^T (c_N h_2(\phi_N), v)_{L^2} \eta    \dt \to \int_0^T ( c h_2(\phi),v)_{L^2} \eta   \dt.
		$$
		Next, we observe that it holds $|h_3(\phi_N)+\partial_\phi f_2(\phi_N)| \lesssim 1 + \abs{\phi_N} + \abs{\phi_N}^3$, and $\phi_N$ is bounded in $L^4(0,T; L^4)$, thus $h_3(\phi_N)+\partial_\phi f_2(\phi_N)$ is bounded in $L^{4/3}(0,T; L^{4/3})$. Thus, the Lebesgue dominated convergence theorem gives $h_3(\phi_N)+\partial_\phi f_2(\phi_N) \to h_3(\phi)+\partial_\phi f_2(\phi)$ strongly in $L^{4/3}(0,T; L^{4/3})$.
		Since $v \in V_\delta \subset L^4(\Omega)$ by \eqref{assump:comp} and $\eta \in C_c^\infty(0,T)$, we have $v \eta \in L^4(0,T; L^4)$. Thus, the product $(h_3(\phi_N)+\partial_\phi f_2(\phi_N)) v \eta$ converges in $L^1(0,T; L^1)$.
		Combining the parts, we finally conclude: 
		$$\int_0^T \left( \partial_\phi f(\phi_N(t), c_N(t)), v \right)_{L^2} \eta(t)  \dt \to \int_0^T \left( \partial_\phi f(\phi(t), c(t)), v \right)_{L^2} \eta(t)  \dt.$$
	\end{itemize}
	Thus, passing to the limit in \eqref{weakform-galerkintime} shows that $(\phi,c)$ solves the variational form of the nonlocal Allen--Cahn/Cahn--Hilliard system as stated in Definition~\ref{Def:Sol}. \smallskip
	
	\noindent\emph{Step 6 (Attainment of initial conditions).} The strong convergence \eqref{Eq:StrongConv} evaluated at $t=0$ gives $\phi_N(0) \to \phi(0)$ in $V_\delta^*$ and by the definition of $\phi_N(0)$ we further know that $\phi_N(0) \to \phi_0$ in $L^2$ and the uniqueness of limits yields $\phi(0)=\phi_0$ in $V_\delta^*$. Regarding the initial condition of the concentration, we argue similarly and obtain $c(0)=c_0$ in $V_\delta^*$.\smallskip
	
	\noindent\emph{Step 7 (Energy inequality).} The discrete energy inequality as stated in \eqref{L2-Gr}--\eqref{X-bound} transfers to the continuous level by noting that norm are weakly lower semicontinuous. Thus, we obtain \eqref{EnergyIneq} as stated in Theorem~\ref{thm:NLAC}.   \smallskip
	
	\noindent\emph{Step 8 (Stability and uniqueness).}
	As the last step, we prove the stability and uniqueness of weak-strong solutions. As noted in Theorem~\ref{thm:NLAC}, we additionally assume Assumption~\ref{assump:free_energy_density2}. Using the enhanced growth condition \eqref{assump:growthenh}, we can improve the regularity of $\partial_t \phi_N$ in the proof from before. Indeed, we can bound
	$$\begin{aligned}
		\int_0^T \abs{\langle \partial_t \phi_N, \Pi_{V_N} \xi \rangle} \dt  &= \int_0^T \Big|-\alpha_\phi L (\phi_N,\Pi_{V_N} \xi)_{V_\delta} - L (\partial_\phi f(\phi_N,c),\Pi_{V_N} \xi)_{L^2}\Big| \dt  \\
		&\lesssim  \norm{\phi_N}_{L^2(V_\delta)} \|\Pi_{V_N} \xi\|_{L^2(V_\delta)} +  (1+\abs{c_N}+\abs{\phi_N},|\Pi_{V_N} \xi|)_{L^2(L^2)} \\
		&\lesssim \norm{\phi_N}_{L^2(V_\delta)} \|\Pi_{V_N} \xi\|_{L^2(V_\delta)} + \bigl(1+\norm{c_N}_{L^2(L^2)}+\norm{\phi_N}_{L^2(L^2)}\bigr) \|\Pi_{V_N} \xi\|_{L^2(L^2)} \\
		&\lesssim \|\Pi_{V_N} \xi\|_{L^2(V_\delta)}.
	\end{aligned}$$
	Thus, it yields $\partial_t \phi_N \in L^2(0,T;V_\delta^*)$ and this regularity transfers to $\partial_t \phi$. This is especially important in that the variational form is then satisfied for any test function in $L^2(0,T;V_\delta)$ instead of $L^4(0,T;V_\delta)$ and thus choosing the test function $\phi$ is valid.
	
	Let $(\phi_1,c_1)$ and $(\phi_2,c_2)$ be two weak solution tuples with $c_1 \in L^1(0,T;L^\infty)$ and we set
	$\phi:=\phi_1-\phi_2$, $c:=c_1-c_2$.  
	First, we subtract the two instances of \eqref{weakform} governing the phase-field and then test with $\phi \in L^2(0,T;V_\delta)$:
	\[
	\dfrac12\ddt\norm{\phi}_{L^{2}}^{2}
	+\alpha_\phi L \abs{\phi}_{V_\delta}^2
	=-L \bigl(\partial_\phi f(\phi_1,c_1)
	-\partial_\phi f(\phi_2,c_2),\phi\bigr)_{L^2}.
	\]
	Using the joint Lipschitz continuity of $\partial_\phi f$, see Assumption~\ref{assump:free_energy_density2}, and Hölder’s inequality, we obtain 
	\begin{equation} \label{Eq:CritUni}
		\begin{aligned}
			(\partial_\phi f(\phi_1,c_1)
			-\partial_\phi f(\phi_2,c_2), \phi)_{L^2}  &\lesssim \norm{\phi}^2_{L^2} + (\abs{c_1}+\abs{\phi_1},\abs{\phi}^2)_{L^2}+ \norm{c}^2_{L^2} \\
			&\lesssim \norm{\phi}^2_{L^2} + (\norm{c_1}_{L^\infty}+\norm{\phi_1}_{L^\infty}) \norm{\phi}^2_{L^2}+ \norm{c}^2_{L^2}.
		\end{aligned}
	\end{equation}
	Inserting this into the estimate, we deduce that
	\begin{equation}\label{Eq:UniqueEst1}
		\dfrac12\ddt\norm{\phi}_{L^{2}}^{2}
		+\alpha_\phi L\abs{\phi}_{V_\delta}^2
		\lesssim \norm{\phi}^2_{L^2} + \big(\norm{c_1}_{L^\infty}+\norm{\phi_1}_{L^\infty}\big) \norm{\phi}^2_{L^2}+ \norm{c}^2_{L^2}.
	\end{equation}
	Furthermore, we consider the test function $Kc$, $K>0$ to be chosen, in the subtracted weak forms governing the concentrations $c_1$ and $c_2$, which gives
	$$\frac{K}{2} \ddt \norm{c}_{L^2}^2 +\alpha_c MK \norm{\mathcal{L}_\delta c_N}_{L^2}^2+ MK(\partial_c f(\phi_1,c_1)-\partial_c f(\phi_2,c_2),c)_{V_\delta} = 0.$$
	Using the affine linearity of $\partial_c f(\phi,c)$ with regard to $c$, see \eqref{assump:affine}, it holds
	$$(\partial_c f(\phi_1,c_1)-\partial_c f(\phi_2,c_2),c)_{V_\delta} =A\abs{c}_{V_\delta}^2 + (h_1(\phi_1)-h_1(\phi_2),c)_{V_\delta}.$$
	Inserting this into the estimate, we deduce with the Lipschitz continuity of $h_1$ (with Lipschitz constant $L_h$) and Young's inequality that
	$$
	\begin{aligned}\frac{K}{2} \ddt \norm{c}_{L^2}^2 +\alpha_c MK \norm{\mathcal{L}_\delta c_N}_{L^2}^2+ AMK\abs{c}_{V_\delta}^2 &= MK(h_1(\phi_2)-h_1(\phi_1),c)_{V_\delta} \\ &\leq \frac{MKL_h^2}{2A}\abs{\phi}_{V_\delta}^2 + \frac{AMK}{2} \abs{c}_{V_\delta}^2,
	\end{aligned}$$
	and absorbing the second term on the right-hand side gives
	\begin{equation}\label{Eq:UniqueEst2}
		\dfrac{K}{2}\ddt\norm{c}_{L^{2}}^{2}
		+\alpha_c MK \norm{\mathcal{L}_\delta c_N}_{L^2}^2+\frac{AMK}{2} \abs{c}_{V_\delta}^2
		\leq \frac{MKL_h^2}{2A} \abs{\phi}_{V_\delta}^2.
	\end{equation}	
	We add \eqref{Eq:UniqueEst1} and \eqref{Eq:UniqueEst2} and integrate the summed estimate over $[0,t]$, $t \leq T$, which yields
	$$\begin{aligned}
		&\dfrac12\norm{\phi(t)}_{L^{2}}^{2}
		+\alpha_\phi L  \abs{\phi}_{L^2(0,t;V_\delta)}^2 +\dfrac{K}{2}\norm{c(t)}_{L^{2}}^{2}+\alpha_c MK \norm{\mathcal{L}_\delta c_N}_{L^2(0,t;L^2)}^2
		+\frac{AMK}{2} \abs{c}_{L^2(0,t;V_\delta)}^2
		\\ &\leq \dfrac12\|\phi(0)\|_{L^{2}}^{2}+\dfrac{K}{2}\|c(0)\|_{L^{2}}^{2}+C\norm{\phi}^2_{L^2(0,t;L^2)} + C\int_0^t \big(\|c_1(s)\|_{L^\infty}+\|\phi_1(s)\|_{L^\infty}\big) \|\phi(s)\|^2_{L^2} \ds\\
		&+ C\norm{c}^2_{L^2(0,t;L^2)} + \frac{MKL_h^2}{2A} \abs{\phi}_{L^2(0,t;V_\delta)}^2.
	\end{aligned}$$
	At this point, we choose $K>0$ sufficiently small in the sense
	$\alpha_\phi L>\frac{MKL_h^2}{2A}$ so that the last term on the right-hand side can be absorbed. Dropping the irrelevant terms and defining $y(t)=\norm{\phi(t)}_{L^2}^2 + K \norm{c(t)}_{L^2}^2$, we have
	$$
	y(t) \lesssim y(0)+\int_0^t k(s) y(s)\ds,
	$$
	and $k$ is integrable as we assumed $c_1,\phi_1 \in L^1(0,T;L^\infty)$. Thus, the Gronwall--Bellman inequality \cite[Lemma II.4.10]{boyer2012mathematical} yields $y(t) \lesssim y(0)$ for any $t \geq 0$. This yields the continuous dependence on initial data and choosing $y(0)=0$ gives $y(t)=0$, which means that the solution is unique for the same data.
\end{proof}

\begin{remark}
	We prove weak--strong uniqueness of solutions as stated in Theorem~\ref{thm:NLAC}, an approach that has become popular in the analysis of nonlinear PDEs (including nonlocal Cahn--Hilliard-type systems; see \cite{brunk2023stability,giorgini2019uniqueness,huo2023existence,frigeri2016nonlocal}). Instead, we could have also opted to prove a small-time or small-data uniqueness result. Indeed, the critical estimate was in \eqref{Eq:CritUni}, where we estimated $c_1$ and $\phi_1$ by their $L^\infty$-norm, but we could instead use H\"older's inequality to infer
	$$(\partial_\phi f(\phi_1,c_1)
	-\partial_\phi f(\phi_2,c_2), \phi)_{L^2}  \lesssim \norm{\phi}^2_{L^2} + (\norm{c_1}_{L^2}+\norm{\phi_1}_{L^2}) \norm{\phi}^2_{L^4}+ \norm{c}^2_{L^2}.$$
	Then, employing an embedding of the form $V_\delta \subset L^4$, we could have absorbed the critical term and obtained
	$$\dots + \bigl(\alpha_\phi L - C\norm{c_1}_{L^\infty(L^2)}-C\norm{\phi_1}_{L^\infty(L^2)}\bigr)\norm{\phi}_{L^2(V_\delta)}^2 \leq \dots,$$
	and the bracket becomes nonnegative if we assume either small time or small data since $c_1$ and $\phi_1$ follow the energy inequality, that is, it holds
	$$\norm{c_1}_{L^\infty(L^2)}+\norm{\phi_1}_{L^\infty(L^2)} \lesssim Ce^{CT} \bigl(\|c_1(0)\|_{L^2} +  \|\phi_1(0)\|_{L^2}\bigr).$$
\end{remark}

\section{Nonlocal-to-local limit convergence}
\label{sec:limit}

In this section, we rigorously show -- under suitable smoothness assumptions on the phase-field and the concentration -- that the nonlocal operators in our model converge to their classical local counterparts as the interaction radius $\delta$ tends to zero. We mention the works \cite{abels2024strong,abels2024acflow,elbar2023degenerate,davoli2021nonlocal,trussardi2019nonlocal} for nonlocal-to-local convergence results for phase-field models. Furthermore, the work \cite{pelech2019peridynamic} studied the nonlocal-to-local convergence for a peridynamic model. In what follows, \emph{error} denotes the difference between the nonlocal solution and the local limit as $\delta\to0$. In the following, we assume that $\Omega$ is a $C^3$-domain. 

\subsection{Convergence of the nonlocal operator for the phase-field}
\label{sec:convergence_nonlocal_operator}
Let $\phi$ be sufficiently regular in the sense $\phi(t)\in C^3(\overline\Omega)$ with
$\abs{\phi(t)}_{W^{3,\infty}} < \infty$ a.e. in $(0,T)$.
Furthermore, we assume that $J_\delta$ is the top-hat kernel as discussed in Example~\ref{Ex:Kernels}, fulfilling the properties \eqref{assump:nonneg}--\eqref{assump:2ndnorm}.
Especially, the kernel is second-order normalized, see \eqref{assump:2ndnorm}, and by isotropy it holds
\[
\int_{B_{\delta}(0)} J_{\delta}(z)z_iz_j\dz
= \frac{1}{d}\delta_{ij}\int_{B_{\delta}(0)}J_{\delta}(z)\lvert z\rvert^2\dz
= 2\delta_{ij}.
\]
For each fixed $x\in\Omega$ and $y\in B_{\delta}(x)$, we define $z = y-x$.  A second-order Taylor expansion gives
\[
\phi(x+z)
= \phi(x)
+ \nabla\phi(x)\cdot z
+ \frac12z^TD^2\phi(x)z
+ \frac{1}{6}\sum_{i,j,k} \partial_{i j k}\phi\bigl(x+\theta z\bigr)z_iz_jz_k,
\]
for some $\theta=\theta(x,z)\in(0,1)$.  Hence, we obtain
\[
\phi(x) - \phi(x+z)
= -\nabla\phi(x)\cdot z
- \frac12z^TD^2\phi(x)z
- \frac{1}{6}\sum_{i,j,k} \partial_{i j k}\phi\bigl(x+\theta z\bigr)z_iz_jz_k.
\]
Furthermore, we define the restricted operator
\[
L_\delta\phi(x)=\alpha_\phi\int_{B_\delta(0)}J_\delta(z)\bigl[\phi(x)-\phi(x+z)\bigr]\dz.
\]
\begin{lemma}[Dual-norm convergence of the nonlocal phase operator]\label{lem:dual_norm}
	If $\phi\in H^3(\Omega)$ and $\phi$ is extended by zero outside $\Omega$, then there is a constant $C>0$, independent of $\delta$, such that
	\[
	\bigl\|L_\delta\phi + \alpha_\phi\Delta\phi\bigr\|_{H^{-1}}
	\le C\delta\norm{\phi}_{H^3}.
	\]
\end{lemma}

\begin{proof}
	We fix an arbitrary $v\in H_0^1(\Omega)$ with $\norm{v}_{H^1}=1$.  We have to bound
	\[
	\langle L_\delta\phi + \alpha_\phi\Delta\phi,v\rangle
	= \alpha_\phi\int_\Omega\int_{B_\delta(0)}
	J_\delta(z)\bigl[\phi(x)-\phi(x+z)\bigr]v(x)\dz\dx
	+\alpha_\phi\int_\Omega\nabla\phi\cdot\nabla v\dx.
	\]
	We change variables in the double integral and add-subtract the Taylor expansion of $\phi(x+z)$ around $x$ to deduce
	\[
	\phi(x+z)=\phi(x)+\nabla\phi(x)\cdot z+\frac12z^TD^2\phi(x)z
	+\frac16\sum_{|\beta|=3}\partial^\beta\phi(x+\theta z)z^\beta, \quad \theta=\theta(x,z)\in(0,1),
	\]
	where $\beta \in \mathbb{N}_0^d$ is a multindex and we use the standard notations of $z^\beta=z^{\beta_1}\cdot \cdot \cdot z^{\beta_n}$ and $\partial^\beta=\partial^{\beta_1} \dots \partial^{\beta_n}$.
	We substitute this formulation into $L_\delta\phi$ to obtain
	\[
	L_\delta\phi(x)
	=-\alpha_\phi\int_{B_\delta(0)} J_\delta(z)\Bigl[\nabla\phi(x)\cdot z
	+\frac12z^TD^2\phi(x)z
	+\mathcal R(x,z)\Bigr]\dz,
	\]
	where we defined the remainder
	\[
	\mathcal R(x,z)
	=\frac16\sum_{|\beta|=3}\partial^\beta\phi(x+\theta z)z^\beta \quad \theta \in (0,1).
	\]
	By symmetry, we deduce that
	\[
	\int_{B_\delta(0)} J_\delta(z)\nabla\phi(x)\cdot z\dz=0,
	\qquad
	\int_{B_\delta(0)} J_\delta(z)z_i z_j\dz=\frac2d\delta_{ij},
	\]
	from which we conclude
	\[
	L_\delta\phi(x)
	=-\alpha_\phi\Delta\phi(x)
	-
	\alpha_\phi\int_{B_\delta(0)} J_\delta(z)\mathcal R(x,z)\dz.
	\]
	Therefore, we obtain
	\[
	\langle L_\delta\phi + \alpha_\phi\Delta\phi,v\rangle
	= -\alpha_\phi\int_{\Omega}\int_{B_\delta(0)}
	J_\delta(z)\mathcal R(x,z)v(x)\dz\dx.
	\]
	We apply the Cauchy--Schwarz inequality in $(x,z)$ together with the estimate
	\[
	\norm{\mathcal R(\cdot ,z)}_{L^2(\Omega)}\le C\norm{\phi}_{H^3(\Omega)}\abs{z}^3,
	\]
	and the second-order normalization property \eqref{assump:2ndnorm} to get
	\[
	\begin{aligned}
		\abs{\langle L_\delta\phi + \alpha_\phi\Delta\phi,v\rangle}
		&\le\alpha_\phi C 
		\Bigl(\int_{B_\delta(0)} \norm{\mathcal R(\cdot ,z)}_{L^2(\Omega)}^2 \abs{z}^{-2} J_\delta(z)\dz\Bigr)^{1/2}
		\Bigl(\int_\Omega\int_{B_\delta(0)} J_\delta(z)\abs{z}^2|v(x)|^2\dz\dx\Bigr)^{1/2}
		\\
		&\lesssim \norm{\phi}_{H^3} \abs{z}_\infty 
		\Bigl(\int_{B_\delta(0)} \abs{z}^{2} J_\delta(z)\dz\Bigr)^{1/2}
		\Bigl(\int_\Omega\int_{B_\delta(0)} J_\delta(z)\abs{z}^2|v(x)|^2\dz\dx\Bigr)^{1/2}
		\\ &\lesssim \delta\norm{\phi}_{H^3}\norm{v}_{H^1(\Omega)}.
	\end{aligned}
	\]
	Finally, the supremum over $\norm{v}_{H^1}=1$ yields the claimed $H^{-1}$-bound.
\end{proof}

\begin{proposition}[Weak solution convergence]\label{prop:weak_limit}
	Let $(\phi_\delta,c_\delta)$ be the  solutions of the nonlocal KKS model \eqref{Eq:ModelNonlocalAC}--\eqref{Eq:ModelNonlocalCH} with spatial regularity
	\[
	\phi_\delta \in
	C^3(\overline\Omega),\quad g_\delta := \partial_c f(\phi_\delta,c_\delta)
	\in
	C^3(\overline\Omega),
	\qquad
	\sup_{x\in\overline\Omega}\max_{|\alpha|\le3}
	\bigl|\partial^\alpha\phi_\delta(x)\bigr|
	+
	\bigl|\partial^\alpha g_\delta(x)\bigr|
	< \infty.
	\] 
	Then as $\delta\to0$, it yields
	\[
	L_\delta\phi_\delta
	\to
	-\alpha_\phi\Delta\phi
	\quad\text{in }L^2(0,T;H^{-1}),
	\quad
	\mathcal{L}_\delta \partial_c f (\phi_\delta,c_\delta)
	\to
	\Delta \partial_c f (\phi,c)
	\quad\text{in }L^2(0,T;H^{-1}).
	\]
	Consequently, any weak-limit point $(\phi,c)$ satisfies the classical local KKS system in the $H^{-1}$-weak sense.
\end{proposition}

\begin{proof}
	By Lemma~\ref{lem:dual_norm}, for each \(t\),
	\[
	\|L_\delta\phi_\delta(t)+\alpha_\phi\Delta\phi_\delta(t)\|_{H^{-1}} \le C\delta\|\phi_\delta(t)\|_{H^3}.
	\]
	The uniform \(H^3\)-bounds and integration in time yield:
	\[
	L_\delta\phi_\delta + \alpha_\phi\Delta\phi_\delta \to 0 \quad \text{in } L^2(0,T;H^{-1}).
	\]
	Since \(\phi_\delta \rightharpoonup \phi\) weakly in \(L^2(0,T; H^3)\), we have \(\Delta \phi_\delta \rightharpoonup \Delta \phi\) weakly in \(L^2(0,T; H^1)\). Hence, for any test function \(\psi \in L^2(0,T; H^1_0)\):
	\[
	\langle L_\delta \phi_\delta, \psi \rangle = - \alpha_\phi \langle \Delta \phi_\delta, \psi \rangle + \langle L_\delta \phi_\delta + \alpha_\phi \Delta \phi_\delta, \psi \rangle.
	\]
	The second term tends to 0 by the convergence above, and the first term converges to \(-\alpha_\phi \langle \Delta \phi, \psi \rangle\) by weak convergence. Thus,
	\[
	L_\delta\phi_\delta \to -\alpha_\phi\Delta\phi \quad \text{in } L^2(0,T;H^{-1}).
	\]
	The identical argument applies to \(\mathcal{L}_\delta \partial_c f(\phi_\delta,c_\delta)\). Passing to limits in the nonlocal weak form yields the classical local weak formulation.
\end{proof}
\ \\
Consequently, the nonlocal Allen--Cahn equation
\[
\partial_t \phi_\delta
= -L\Bigl\{ \partial_\phi f(\phi,c)
+ \alpha_{\phi}\int_{B_{\delta}(x)}J_{\delta}(x-y)[\phi(x)-\phi(y)]\dy
\Bigr\}
\]
converges, as $\delta\to0$, to the classical local Allen--Cahn equation
\[
\frac{\partial \phi}{\partial t}
= -L\Bigl\{ \partial_\phi f(\phi,c) 
- \alpha_{\phi}\Delta\phi \Bigr\},
\]
with error $\mathcal{O}(\delta)$ in $L^\infty(\Omega)$.

\subsection{Convergence of the nonlocal operator for the concentration evolution}
\label{sec:convergence_nonlocal_concentration}

Next, we consider the nonlocal concentration evolution,
\[
\partial_t c
= M \int_{B_{\delta}(x)} J_{\delta}(x-y)\Bigl[
g(y) - g(x)\Bigr]\dy,
\quad
g(x) := \partial_c f\bigl(\phi(x),c(x)\bigr).
\]
We assume $g\in C^3(\overline{\Omega})$ with $|g|_{W^{3,\infty}}<\infty$ and
for each fixed $x$ and $y\in B_{\delta}(x)$, we define $z=y-x$. Again, using Taylor's extension gives
\[
g(x+z)
= g(x)
+ \nabla g(x)\cdot z
+ \frac12z^TD^2 g(x)z
+ \frac{1}{6}\sum_{i,j,k} \partial_{i j k}g\bigl(x+\theta z\bigr)z_iz_jz_k, \quad \theta=\theta(x,z)\in(0,1).
\]
Then
\[
g(x+z)-g(x)
= \nabla g(x)\cdot z
+ \frac12z^TD^2 g(x)z
+ \frac{1}{6}\sum_{i,j,k} \partial_{i j k}g\bigl(x+\theta z\bigr)z_iz_jz_k.
\]
Since $J_{\delta}$ is radial, the linear term integrates to zero and the cubic-order term also vanishes by symmetry.  Therefore
\[
\int_{B_{\delta}(0)} J_{\delta}(z)[g(x+z)-g(x)]\dz
= \frac12 \int_{B_{\delta}(0)} J_{\delta}(z)z^TD^2 g(x)z \dz
+ R_{\delta}^g(x),
\]
where the exact error is
\[
R_{\delta}^g(x)
= \frac{1}{6}\int_{B_{\delta}(0)}J_{\delta}(z)
\sum_{i,j,k}\partial_{i j k}g\bigl(x+\theta z\bigr)z_iz_jz_k \dz,
\]
and
\[
\bigl|R_{\delta}^g(x)\bigr|
\le \frac{1}{6}\Bigl(\sup_{|\alpha|=3}\|\partial^{\alpha}g\|_{L^{\infty}}\Bigr)
\int_{B_{\delta}(0)}J_{\delta}(z)\lvert z\rvert^3\dz
= \mathcal{O}(\delta).
\]
By the same isotropy argument, we deduce
\[
\int_{B_{\delta}(0)} J_{\delta}(z)z^TD^2 g(x)z \dz
= \sum_{i,j} \partial_{i j} g(x)\int_{B_{\delta}(0)}J_{\delta}(z)z_iz_j\dz
= 2\Delta g(x).
\]
Hence, we conclude
\[
\int_{B_{\delta}(x)} J_{\delta}(x-y)\bigl[g(y)-g(x)\bigr]\dy
= \Delta g(x) + \widetilde R_{\delta}(x),
\quad
\|\widetilde R_{\delta}\|_{L^{\infty}(\Omega)} = \mathcal{O}(\delta).
\]
It follows that
\[
\frac{\partial c}{\partial t}(x,t)
= M\Delta g(x) + \mathcal{O}(\delta)
=
M\Delta\Bigl[\partial_c f(\phi,c)\Bigr] + \mathcal{O}(\delta),
\]
in $L^\infty(\Omega)$. 
Equivalently, as $\delta\to0$,
\[
\frac{\partial c}{\partial t}(x,t)
\longrightarrow
\nabla\bigl(M\nabla[\partial_c f(\phi,c)]\bigr)(x,t),
\]
with error $\mathcal{O}(\delta)$ in $L^\infty$.  In particular, no extra $\delta^{-2}$ rescaling is needed once the kernel is normalized as above.
\begin{remark}[Weak solution convergence]
	Although we assumed $C^3$-regularity for pointwise $\mathcal O(\delta)$ error, one can show for general weak solutions that
	\[
	L_\delta\phi \to -\alpha_\phi\Delta\phi
	\quad\text{and}\quad
	\mathcal{L}_\delta\mu \to \Delta\mu
	\quad\text{in }H^{-1}(\Omega)
	\]
	as \(\delta\to0\).  Consequently, the nonlocal gradient-flow system converges in the weak $H^{-1}$-sense to the classical KKS equations without requiring $C^3$ hypotheses.
\end{remark}

\begin{remark}[Sharp-interface limit and peridynamic diffusion] \label{sec:structural_equivalence} 
	We want to comment on the scenario when one fixes the macroscopic interaction horizon
	$\delta>0$ and lets the diffuse-interface width
	$\varepsilon\in(0,\delta/2)$ tend to zero. A rigorous justification of this formal limit for related nonlocal phase-field models can be found in \cite{burkovska2021nonlocal,seleson2013interface} and for local phase-field model in \cite{noldner2023sharp}.  The $\eps$-dependence is introduced only through the
	double-well part of the KKS free energy density as discussed in Example~\ref{Ex:KKSdensity},
	\[
	f_\varepsilon(\phi,c)
	:= A\bigl[c-h(\phi)(1-c_L)-c_L\bigr]^{2}
	+\frac{\omega}{\varepsilon}g(\phi).
	\]
	All limits below are taken with \(\varepsilon\to0\) at \emph{fixed}
	\(\delta\). The $\Gamma$-convergence and PDE limits are therefore taken as
	$\varepsilon\to0$ with the macroscopic horizon $\delta>0$ kept
	fixed; $\delta$ enters only through the operator $\mathcal L_\delta$ that
	appears in the peridynamic law. We also assume well-prepared initial data: there is a smooth open set
	$\Omega_{\mathrm m}^{0}\Subset\Omega$ such that
	\[
	\|\phi_\varepsilon(\cdot,0)-\chi_{\Omega_{\mathrm m}^{0}}\|_{C^{1}(\Omega)}
	\le C\varepsilon,
	\qquad
	\|c_\varepsilon(\cdot,0)\|_{H^{1}(\Omega)}\le C,
	\]
	for a constant $C$ independent of $\varepsilon$.
	Finally, we impose an interface-boundary separation
	$\operatorname{dist}(\Gamma_\varepsilon(t),\partial\Omega)\ge
	d_0>0$ (uniform in $t$) to avoid horizon truncation; see
	Remark~\ref{rem:boundary-layer} below.
	Under these hypotheses, the nonlocal Allen--Cahn solution
	$\phi_\varepsilon$ converges in $L^{1}(\Omega)$, up to a subsequence, to
	the characteristic function $\chi_{\Omega_{\mathrm m}(t)}$ of an
	evolving metal domain.  Then one proceeds to show $\phi_\varepsilon\to\chi_{\Omega_{\mathrm m}(t)}$ in $L^{1}(\Omega)$,
	while the $\varepsilon$-independent peridynamic term is
	retained in the limit. Formally let
	\[
	\mu_\varepsilon
	:=\partial_c f\bigl(\phi_\varepsilon,c_\varepsilon\bigr)
	=2A\Bigl[c_\varepsilon-c_L-(1-c_L)h(\phi_\varepsilon)\Bigr]
	=:2A\bigl[c_\varepsilon-c_L\bigr]+\vartheta_\varepsilon,
	\qquad
	\vartheta_\varepsilon:=-2A(1-c_L)h(\phi_\varepsilon).
	\]
	
	\medskip
	\noindent
	Because \(h(\phi_\varepsilon)\to h(\phi_0)\in\{0,1\}\) a.e., the limit
	\(
	\vartheta_0(x,t)\in\{-2A(1-c_L),0\}
	\)
	is piecewise constant.  The horizon operator annihilates constants, so
	\[
	\mathcal L_\delta u(x)
	:=\int_{\Omega}J_{\delta}(x-y)[u(x)-u(y)]\dy,
	\quad \mathcal L_\delta[\vartheta_0]=0,\quad\text{inside each pure phase}.
	\] 
	Indeed, for any constant $k$,
	$\mathcal L_\delta k(x)=\int_{B_\delta(x)} J_\delta(x-y)[k-k]\dy=0$.
	Hence $\mathcal L_\delta[\vartheta_0]$ is supported only in the
	$\delta$-strip 
	\[
	\mathcal N_\delta(t)
	:=\bigl\{x\in\Omega:
	\operatorname{dist}\bigl(x,\partial\Omega_{\mathrm m}(t)\bigr)
	<\delta\bigr\}.
	\]
	Since it holds \(\mathcal L_\delta[\vartheta_0](x)=O(\delta^{-1})\) on
	\(\mathcal N_\delta(t)\) and \(|\mathcal N_\delta(t)|=O(\delta)\), one can show $
	\|\mathcal L_\delta[\vartheta_0]\|_{L^{2}(\Omega)}
	= O(\delta^{-1/2}).
	$
	If we formulate the concentration equation on the
	\emph{bulk domain}  
	\(\displaystyle\Omega_\delta(t):=\Omega\setminus\mathcal N_\delta(t)\),
	in fact
	\(\mathcal L_\delta[\vartheta_0](x,t)\equiv0\) on
	\(\Omega_\delta(t)\), so we obtain the peridynamic diffusion law
	\begin{equation}
		\partial_t c
		= 2AM\mathcal L_\delta[c]
		\qquad\text{in } \Omega_\delta(t),
		\label{eq:pd-final}
	\end{equation}
	supplemented with interface conditions that follow from the sharp-interface
	limit of the phase-field. Testing the nonlocal Cahn--Hilliard equation with a smooth $\psi=\psi(x,t)$, passing to the limit $\varepsilon\to0$, and using the
	time-derivative bound established in Section~\ref{sec:wellposed_analysis}, we arrive formally at the peridynamic diffusion law \eqref{eq:pd-final} displayed above. The formal chain of arguments in this section can be summarised as
	\[
	\text{nonlocal KKS }(\varepsilon,\delta)
	\xRightarrow[\varepsilon\to0]{\text{sharp  interface}}
	\text{peridynamic diffusion }(\delta).
	\]
	In peridynamic diffusion, including corrosion and heat-conduction models
	\cite{bobaru2012heat,zeleke2021review}, one writes
	$$\partial_t C=\int_{B_\delta(x)}\lambda(x,y){\cal J}(x-y)[C(y)-C(x)]\dy,$$
	with a unit-mass influence function ${\cal J}$ and fitted diffusivities
	$\lambda_{\mathrm m}$ (metal) and $\lambda_{\mathrm e}$ (electrolyte).
	Choosing ${\cal J}=J_\delta$ makes (\ref{eq:pd-final}) identical to that
	format provided one sets
	\[
	\lambda_{\mathrm m}=2A(1-c_L),
	\qquad
	\lambda_{\mathrm e}=2Ac_L.
	\]
	Thus all material parameters are inherited from the bulk free energy of
	the phase-field model; the horizon $\delta$ is the only new
	peridynamic-specific length.
\end{remark}

\begin{remark}[Boundary layer]\label{rem:boundary-layer}
	If \(\Gamma_\varepsilon\) hits \(\partial\Omega\) so that
	\(\operatorname{dist}(\Gamma_\varepsilon,\partial\Omega)=0\), the part of
	the horizon \(B_\delta(x)\) lying outside \(\Omega\) produces an extra
	\(H^{-1}\) term bounded by \(C\delta\).  Because \(\delta\) is fixed in
	this section, that contribution is uniformly bounded and can be absorbed
	into the \(O(\varepsilon)\) error of the sharp-interface limit.
\end{remark}


\section{Numerical solution}\label{sec:discretization_and_experiments}

There are several candidates for discretization schemes used with truncated nonlocal operators \cite{delia2020numerical}. While meshfree methods (such as PeriPy) currently dominate practical applications of nonlocal models in mechanics \cite{boys2021peripy,parks2011peridynamics}, both finite element and finite difference methods are widely applied.
As in the local case, finite difference schemes generally require a higher regularity of the solution \cite{tian2013analysis}, while finite element methods offer greater flexibility and are well suited for complex geometries. Beyond these familiar differences, the nonlocal framework introduces additional nuances such as the computation of double integrals, which increases their computational cost. Nonetheless, the nonlocal models often exhibit stabilizing properties in contrast to their local counterpart, as it was also observed in the discretization of a nonlocal Cahn--Hilliard model in \cite{brunk2024analysis}. Furthermore, we mention the works \cite{martowicz2019solving,zhu2025numerical} and the references therein for numerical methods for nonlocal PDEs.

In this work, we consider the explicit Euler time discretization of the KKS phase-field corrosion model and its meshfree approximation of nonlocal integrals via a one-point Gaussian quadrature (peridynamic-style) on a uniform grid. We apply this scheme to a two-dimensional pitting corrosion benchmark in a square domain and numerically demonstrate that, as the nonlocal horizon $\delta\to0$, the nonlocal solution converges in the $L^{2}$-norm to the corresponding local (classical) solution.  

\subsection{Numerical setup}

We discretize the (nondimensional) domain $\Omega = [0,50]\times[0,50]$ using a uniform Cartesian grid with spacing $h=\Delta x=\Delta y$. A circular pit of radius $R=2$ is centered at $(25,25)$; inside the pit we enforce $\phi=0$, $c=0$, while on the remainder of the boundary we impose homogeneous Neumann conditions for both $\phi$ and $c$ (no-flux and zero normal derivative). We treat the pit region as a Dirichlet-layer removal of metal (constant boundary), and the outer boundary as free (Neumann).  

At each grid point $(x_{i},y_{j})$ with $i,j=0,1,\ldots,N$ (where $N=50/h$), we denote by $\phi_{i,j}^{n}$ and $c_{i,j}^{n}$ the phase and concentration at time $t^{n}=n\Delta t$. The local evolution equations (Allen--Cahn for $\phi$, diffusion for $c$) are discretized by forward Euler:
\[
\begin{aligned}
	\phi_{i,j}^{n+1}
	&= \phi_{i,j}^{n} 
	- \Delta t L \Bigl\{\frac{\partial f}{\partial \phi}(\phi_{i,j}^{n},c_{i,j}^{n}) 
	- \alpha_{\phi}\Delta_{\text{FD}}\phi\bigl|_{i,j}^{n}\Bigr\},\\
	c_{i,j}^{n+1}
	&= c_{i,j}^{n}
	+ \Delta t M \Delta_{\text{FD}}\Bigl[\frac{\partial f}{\partial c}(\phi,c)\Bigr]\bigl|_{i,j}^{n},
\end{aligned}
\]
where $\Delta_{\text{FD}}$ is the standard five-point finite-difference Laplacian with homogeneous Neumann (zero-normal-derivative) boundary conditions, clamped to enforce $\phi=0$, $c=0$ inside the pit.  The nonlocal counterpart replaces each Laplacian by a mesh-free sum over neighbors within the horizon $\delta$:
\[
\begin{aligned}
	\phi_{i,j}^{n+1}
	&= \phi_{i,j}^{n} 
	- \Delta t L \Bigl\{\frac{\partial f}{\partial \phi}(\phi_{i,j}^{n},c_{i,j}^{n}) 
	+ \alpha_{\phi}
	\sum_{(k,l)\in\mathcal{N}_{\delta}(i,j)}
	K\bigl[\phi_{i,j}^{n}-\phi_{k,l}^{n}\bigr]h^{2}\Bigr\},\\
	c_{i,j}^{n+1}
	&= c_{i,j}^{n}
	+ \Delta t M \Bigl\{
	\sum_{(k,l)\in\mathcal{N}_{\delta}(i,j)}
	K\Bigl[\frac{\partial f}{\partial c}(\phi_{k,l}^{n},c_{k,l}^{n})
	- \frac{\partial f}{\partial c}(\phi_{i,j}^{n},c_{i,j}^{n})\Bigr]h^{2}
	\Bigr\},
\end{aligned}
\]
where $\mathcal{N}_{\delta}(i,j)=\{(k,l):\|(k-i,k-j)h\|\le\delta\}$, and $K =\frac{8}{\pi\delta^{4}}$ is the proportionality constant for a top-hat kernel in two dimensions. We work entirely with nondimensional variables, using
\[
x'=\frac{x}{L_{c}},\quad t'=\frac{t}{T_{c}},\quad
c'=\frac{c}{C_{c}},\quad\phi'=\phi,
\]
and choose $L_{c}=50$, $T_{c}=1$, $C_{c}=1$.  Table~\ref{tab:params_updated} lists the parameter values used in our numerical simulations.  

\begin{table}[ht]
	\centering
	\caption{Nondimensional parameters in the numerical examples.}
	\label{tab:params_updated}
	\begin{tabular}{l c}
		\hline
		Parameter & Value \\
		\hline
		Kinetic coefficient $L$             & 0.23529   \\
		Free-energy coefficient $A$        & 25.7211   \\
		Equilibrium concentration $c_{L}$   & 0.0357    \\
		Double-well coefficient $\omega$    & 1.0       \\
		\hline
	\end{tabular} \qquad
	\begin{tabular}{l c}
		\hline
		Parameter & Value \\
		\hline
		Nonlocal coefficient $\alpha_{\phi}$& $7.2115\times10^{-7}$ \\
		Mobility $M$                       & 1.945     \\
		Time step $\Delta t$    & $10^{-4}$ \\
		Pit radius $R$                     & 2         \\
		\hline
	\end{tabular}
\end{table}
\ \\
To determine a numerically stable forward–Euler time increment, we use a linear (Fourier/von Neumann) stability bound for the diffusive operators. For the local Laplacian in two space dimensions the classical Courant–Friedrichs–Lewy (CFL) restriction is
\begin{equation}
	\Delta t \;\le\; \frac{h^{2}}{4\,D_{\mathrm{lin}}},
\end{equation}
where the effective diffusivity is taken from the linear parts of the updates,
\[
D_\phi:=L\,\alpha_\phi,
\qquad
D_c:=2A\,M,
\qquad
D_{\mathrm{lin}}:=\max\{D_\phi,D_c\}.
\]
For the nonlocal (peridynamic-type) diffusion operator with horizon $\delta = m h$, a von Neumann analysis yields the less restrictive bound
\begin{equation}
	\Delta t \;\le\; \alpha(m)\,\frac{h^{2}}{D_{\mathrm{lin}}}, 
	\qquad
	\alpha(m)\;=\;\frac{m^{2}}{4\,H_{m}}, 
	\qquad
	H_{m}=\sum_{j=1}^{m}\frac{1}{j},
\end{equation}
which reduces to the classical CFL at $m=1$ and increases monotonically with $m$ \cite{mossaiby2025efficient}.
Initially, we set
\[
\phi(x,y,0)=
\begin{cases}
	0, & \|(x-25,y-25)\|\le2,\\
	1, & \text{otherwise},
\end{cases}
\quad
c(x,y,0)=
\begin{cases}
	0, & \|(x-25,y-25)\|\le2,\\
	1, & \text{otherwise},
\end{cases}
\]
and enforce $\phi=0$, $c=0$ inside the pit at each time step.  On the outer boundary $\partial\Omega$, we impose $\frac{\partial\phi}{\partial n}=\frac{\partial c}{\partial n}=0$. 

\subsection{Numerical simulations}

To illustrate convergence of the nonlocal FD scheme to its local counterpart, we fix $\delta=3h$ and run four pairs of simulations with $h=2,1,0.5,0.25$ (so that $\delta=6,3,1.5,0.75$).  At the final time $t=15$, we extract the solution vectors $\phi_{\text{FD-local}}$ and $\phi_{\text{FD-nonlocal}}$ on each grid and compute the relative $L^{2}$-error
\[
\text{rel}L^{2}\text{-error} 
= 100 \times 
\frac{\|\phi_{\text{FD-local}} - \phi_{\text{FD-nonlocal}}\|_{2}}
{\|\phi_{\text{FD-local}}\|_{2}}\%.
\]
Figure~\ref{fig:conv_L2} plots this percentage error versus $\delta$ on a log-log scale, demonstrating approximately first-order decay in $\delta$.  
\begin{figure}[ht]
	\centering
	\includegraphics[width=0.5\textwidth]{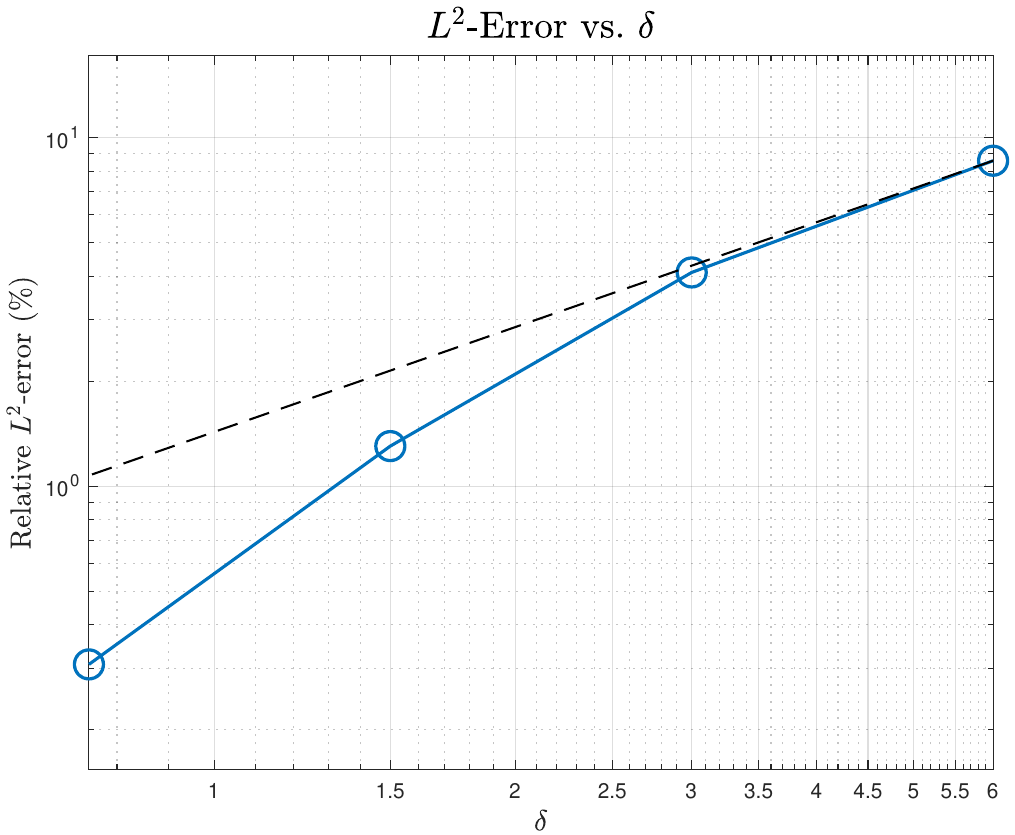}
	\caption{Log-log plot of relative $L^{2}$-error (\%) between local FD and nonlocal FD solutions at $t=15$, versus nonlocal horizon $\delta=3h$ for $h=2,1,0.5,0.25$.  A reference $O(\delta)$ line (dashed) is shown for comparison.}
	\label{fig:conv_L2}
\end{figure}
Finally, we also display the terminal-time phase and concentration fields side by side, to visualize how $\phi$ and $c$ vary with $\delta$ and $h$.  In the first set of panels (Figure~\ref{fig:endpoint_m}), we fix $h=0.25$ and plot $(\phi,c)$ at $t=15$ for $m=2,4,6$ (so $\delta=0.5,1.0,1.5$).  In the second set (Figure~\ref{fig:endpoint_delta}), we fix $\delta=0.75$ and plot $(\phi,c)$ at $t=15$ for $h=0.25,0.5,1.0$ (so $m=3,1.5,0.75$).  
\begin{figure}[htbp!]
	\resizebox{.9\textwidth}{!}{
		\begin{tabular}{cM{.28\textwidth}M{.28\textwidth}M{.28\textwidth}M{.04\textwidth}}
			&$m=2$&$m=4$&$m=6$ \\
			$\phi$ &\includegraphics[height=.22\textheight,clip,trim={35 35 50 25}]{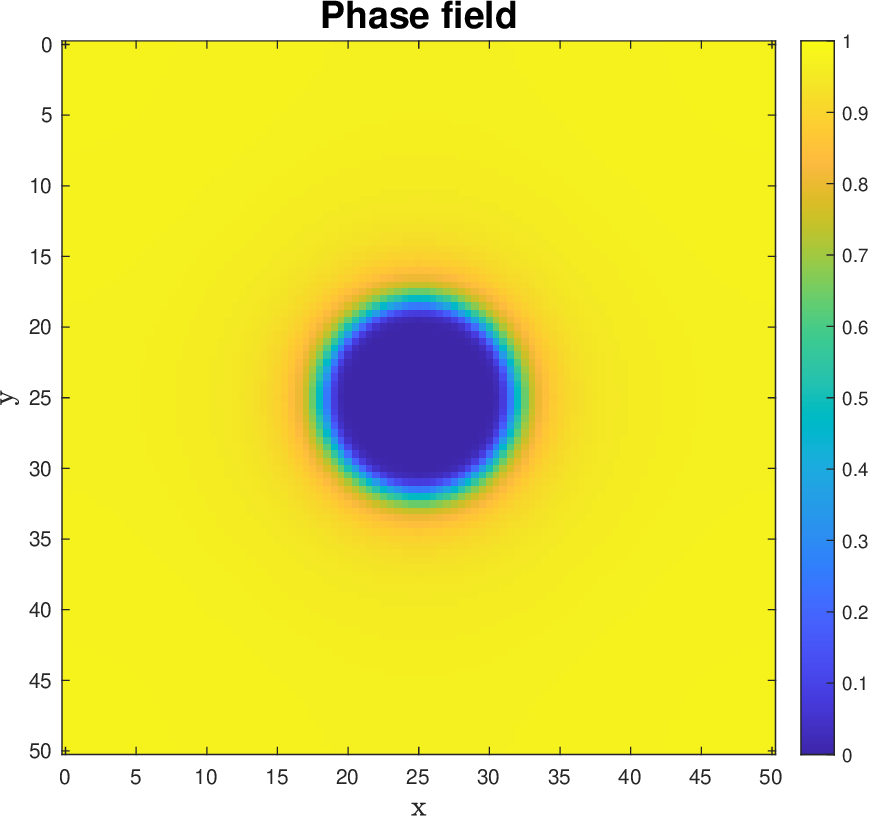} 
			&\includegraphics[height=.22\textheight,clip,trim={35 35 50 25}]{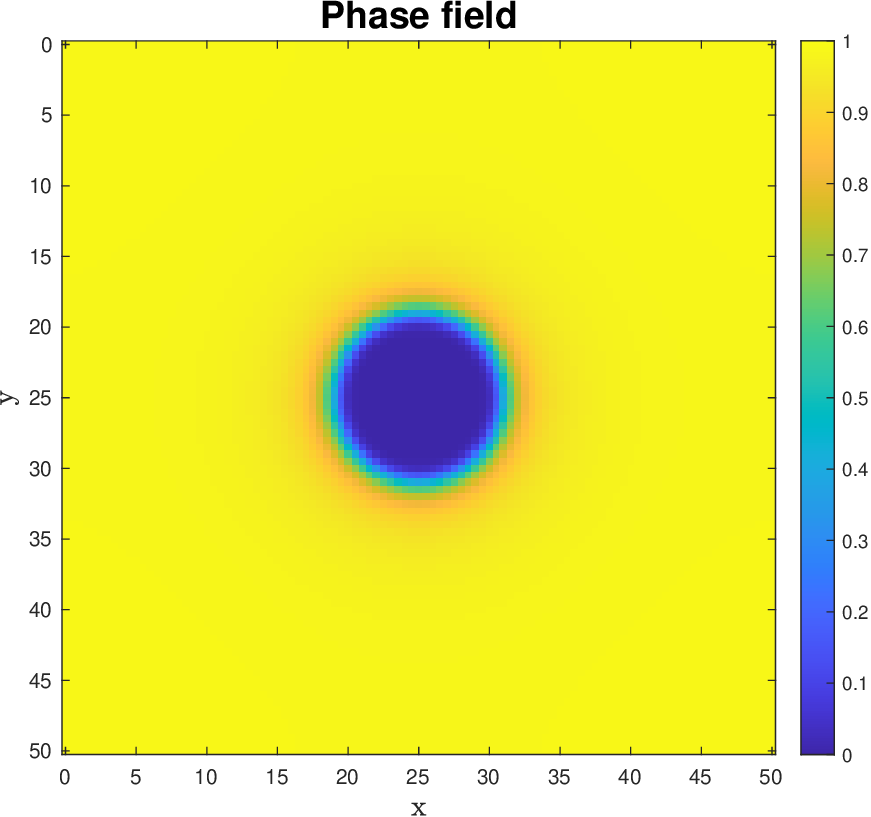}
			&\includegraphics[height=.22\textheight,clip,trim={35 35 50 25}]{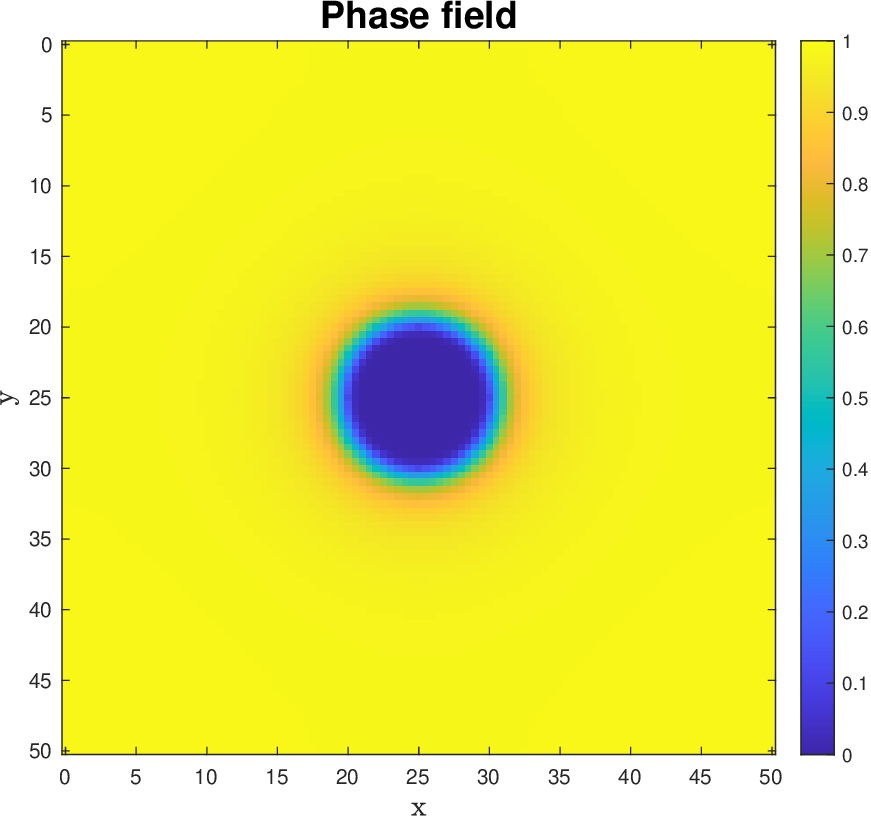} 
			&\includegraphics[height=.22\textheight,width=.7cm]{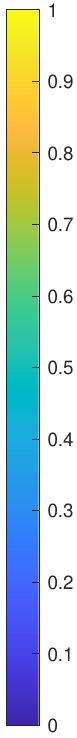}
			\\
			$c$&\includegraphics[height=.22\textheight,clip,trim={35 35 50 25}]{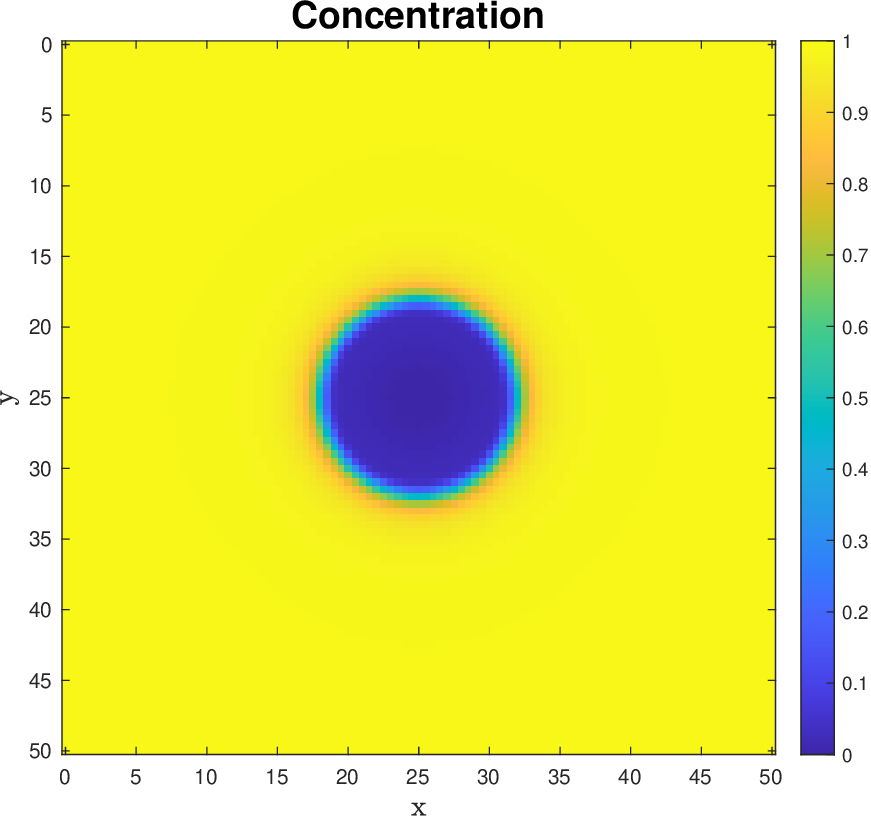} 
			&\includegraphics[height=.22\textheight,clip,trim={35 35 50 25}]{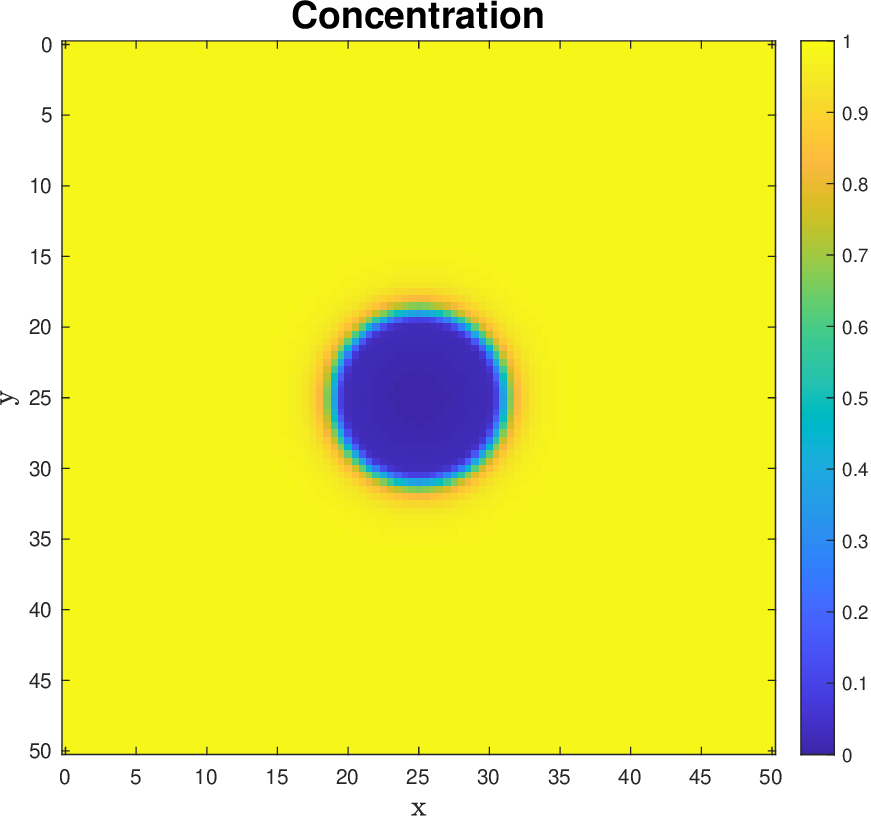}
			&\includegraphics[height=.22\textheight,clip,trim={35 35 50 25}]{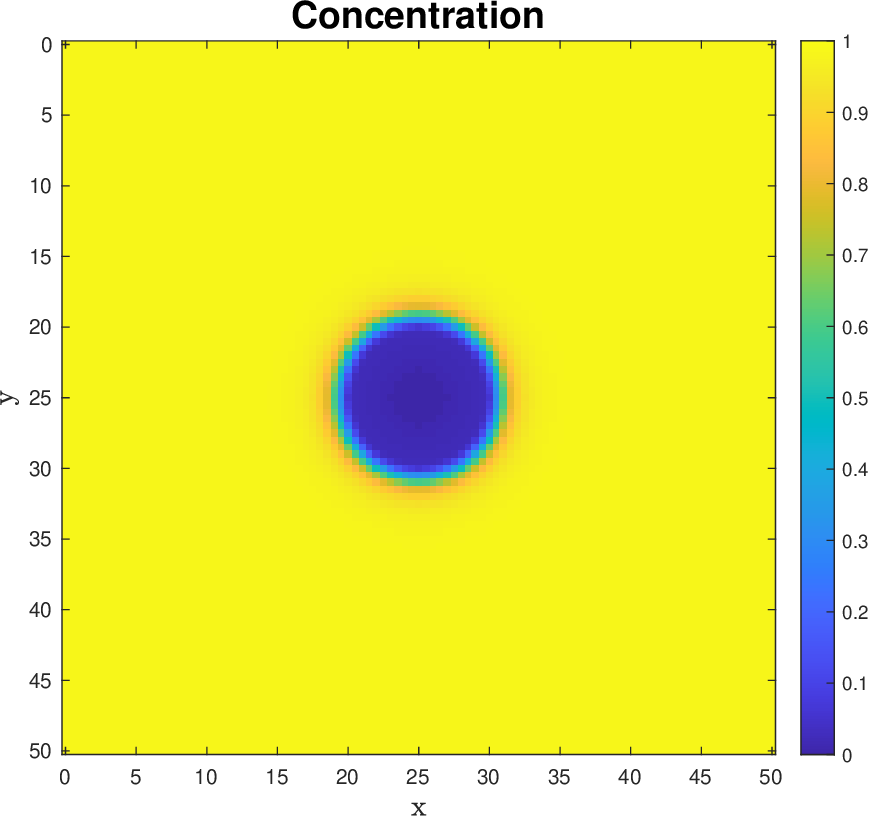} &\includegraphics[height=.22\textheight,width=.7cm]{color.pdf}
	\end{tabular} }
	\caption{Terminal-time ($t=15$) snapshots of $\phi$ (left) and $c$ (right) for $h=0.5$ with $m=\delta / h = 2,4,6$ .}
	\label{fig:endpoint_m}
\end{figure}
\begin{figure}[htbp!]
	\resizebox{.9\textwidth}{!}{
		\begin{tabular}{cM{.28\textwidth}M{.28\textwidth}M{.28\textwidth}M{.04\textwidth}}
			&$h=1$&$h=0.5$&$h=0.25$ \\
			$\phi$ &\includegraphics[height=.22\textheight,clip,trim={35 35 50 25}]{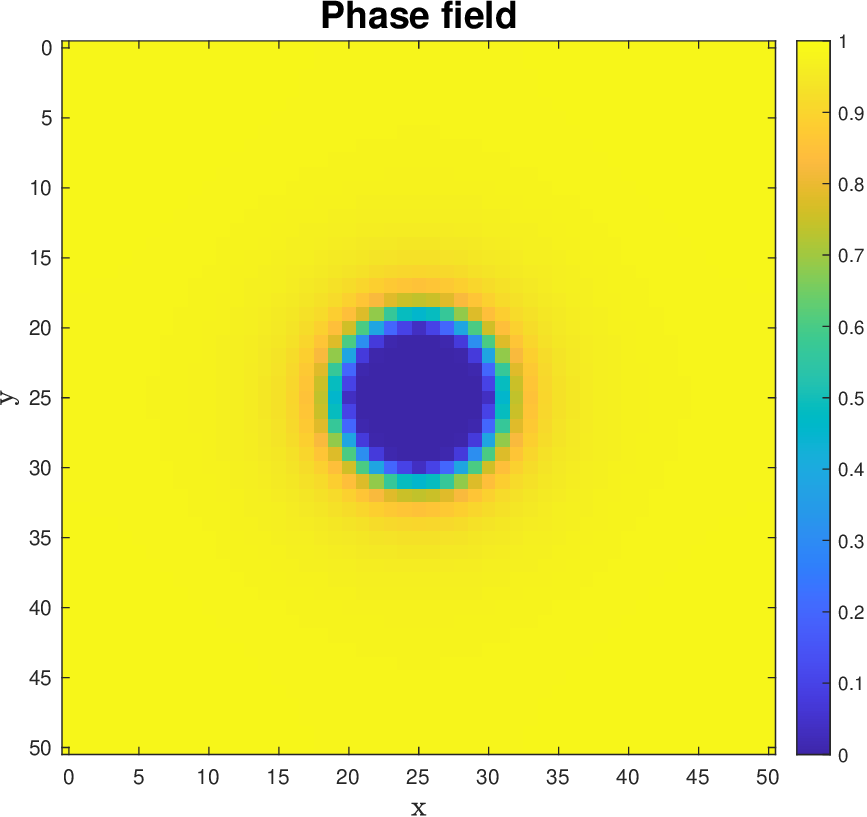} 
			&\includegraphics[height=.22\textheight,clip,trim={35 35 50 25}]{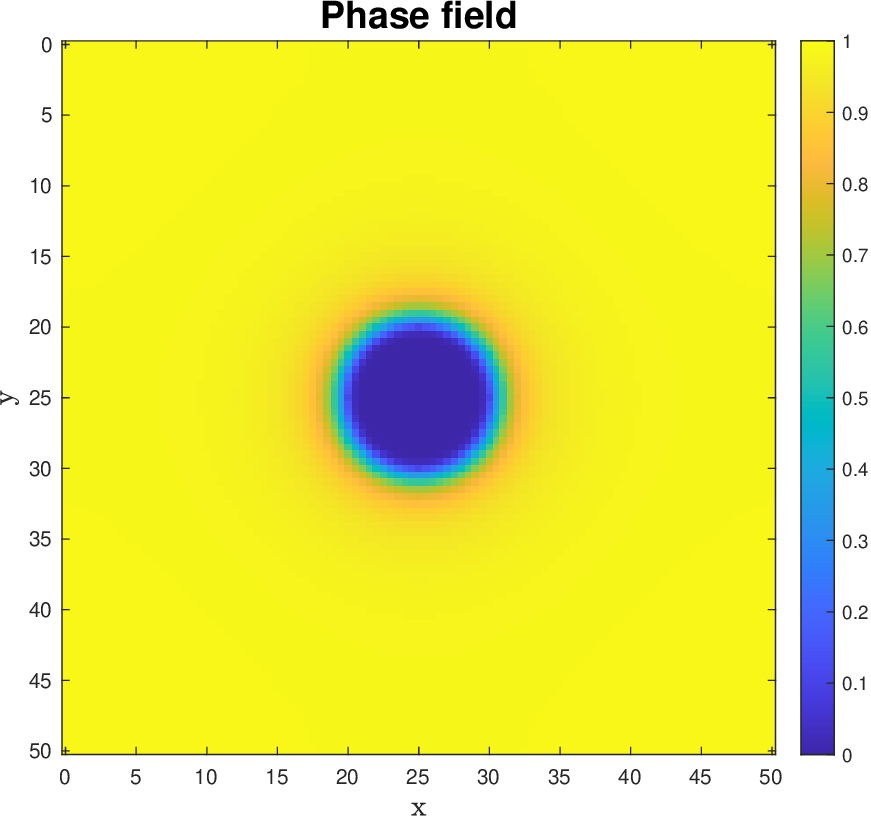}
			&\includegraphics[height=.22\textheight,clip,trim={35 35 50 25}]{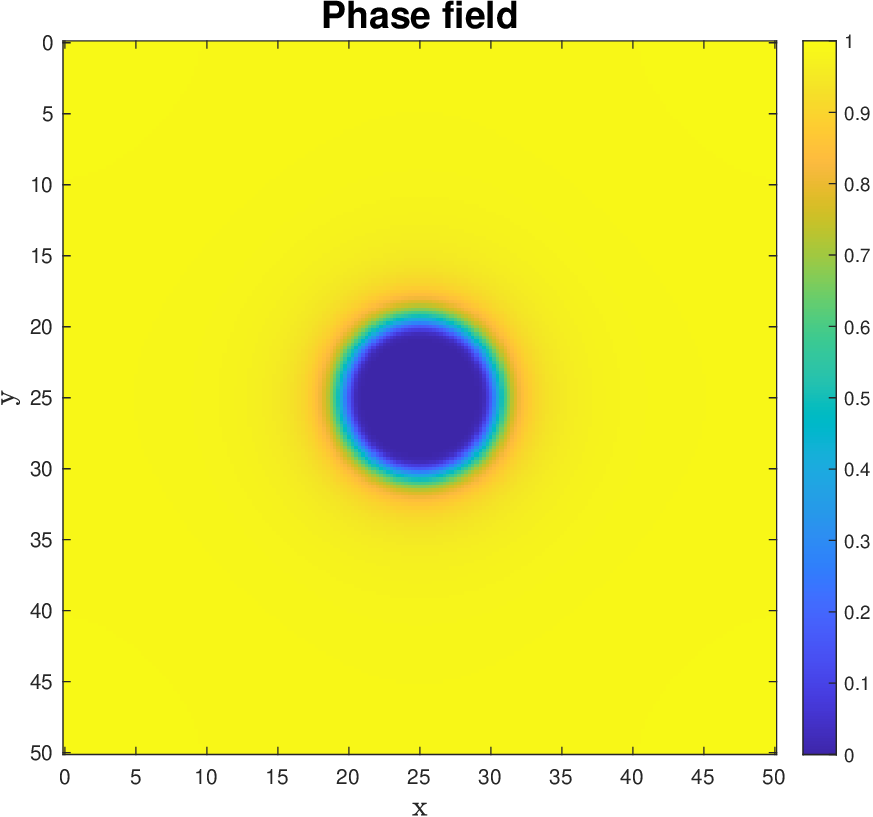} &\includegraphics[height=.22\textheight,width=.7cm]{color.pdf}\\
			$c$&\includegraphics[height=.22\textheight,clip,trim={35 35 50 25}]{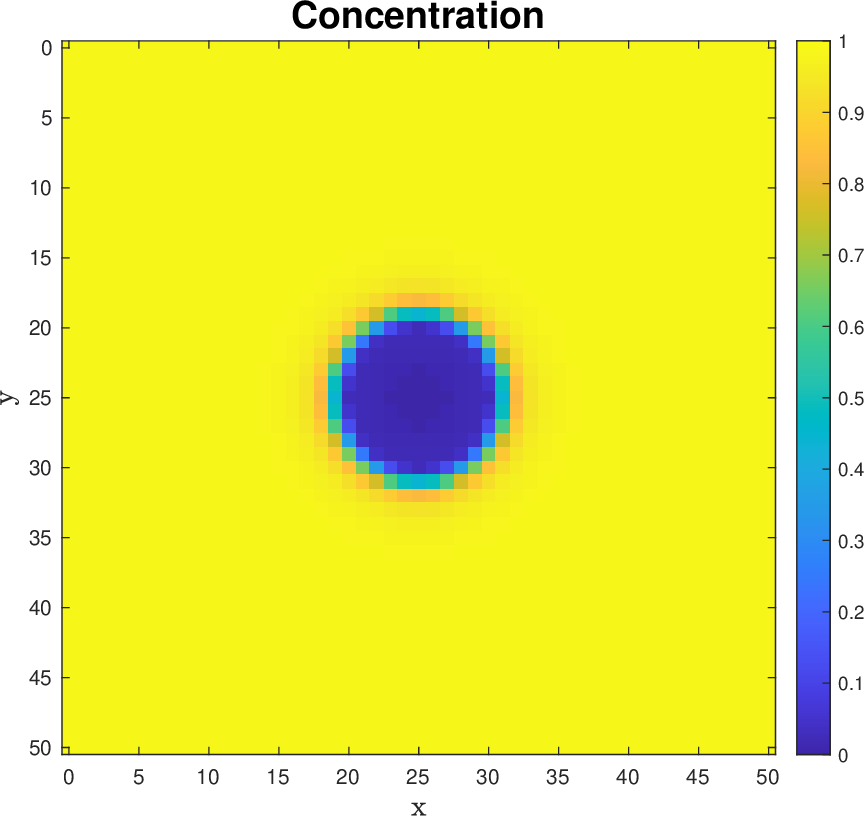} 
			&\includegraphics[height=.22\textheight,clip,trim={35 35 50 25}]{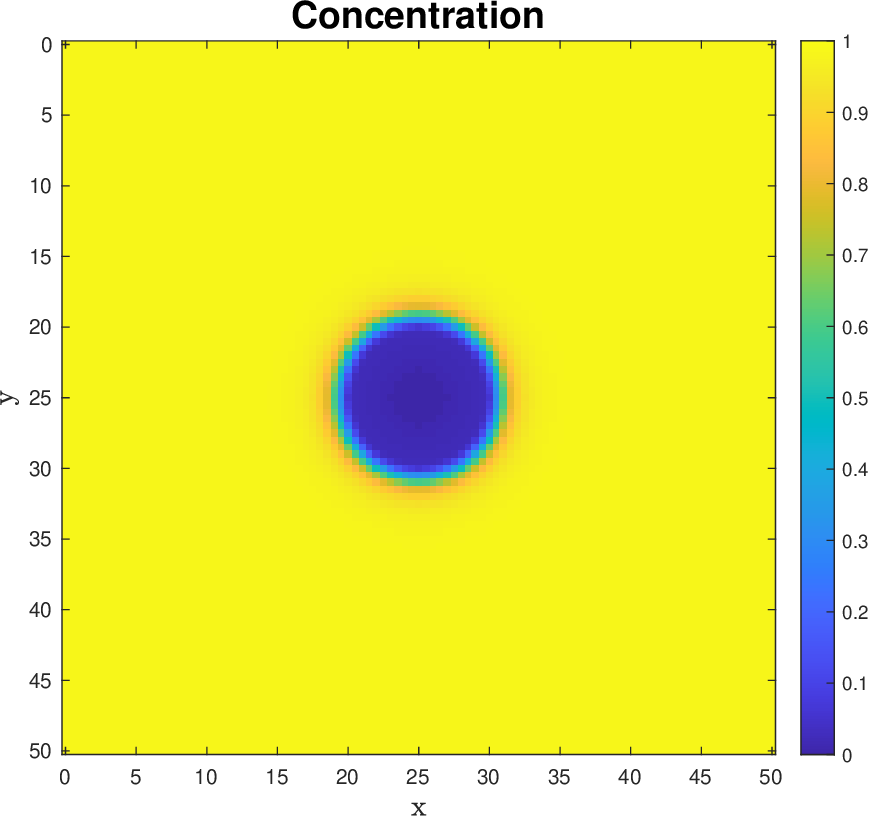}
			&\includegraphics[height=.22\textheight,clip,trim={35 35 50 25}]{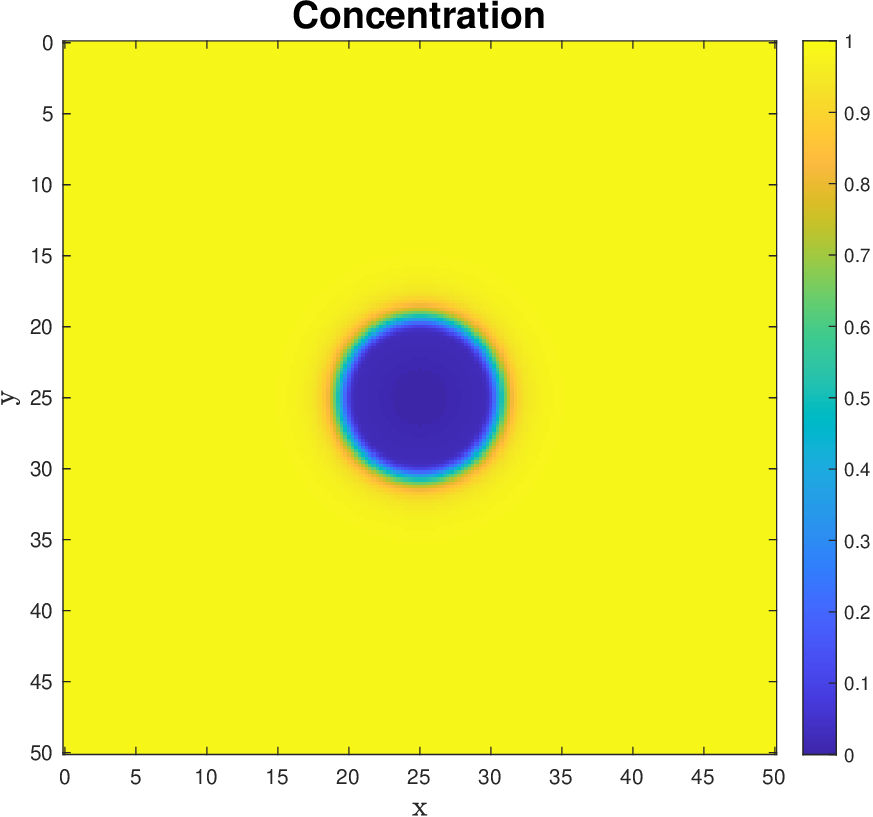} &\includegraphics[height=.22\textheight,width=.7cm]{color.pdf}
	\end{tabular} }
	\caption{Terminal-time ($t=15$) snapshots of $\phi$ (left) and $c$ (right) for $\delta=3$ with $h=1.0,0.5,0.25$.}
	\label{fig:endpoint_delta}
\end{figure}
These visualizations, together with the convergence plot in Figure~\ref{fig:conv_L2}, confirm that as $\delta\to0$ (at fixed $h$) or as $h\to0$ (at fixed $\delta$), the nonlocal FD solution approaches the local FD solution, and illustrate how the interface and concentration fields sharpen and converge in the limit. 

\section{Conclusion and Perspectives}
\label{sec:conclusion}
In this work, we have developed a unified nonlocal framework for phase-field corrosion modeling that naturally bridges to peridynamic approaches. By introducing a finite interaction horizon $\delta$, our model captures long-range chemical interactions and diffusion, yielding a rigorous nonlocal evolution. We have rigorously shown via formal Gateaux differentiation and careful Taylor expansions that, as $\delta \to 0$, the nonlocal evolution equations converge to the classical Allen--Cahn and Cahn--Hilliard limits at an explicit $O(\delta)$ rate. Numerically, using explicit finite-difference discretizations, we have confirmed convergence of the nonlocal solutions to their local counterparts and demonstrated how the horizon $\delta$
and grid spacing $h$ jointly govern pit growth and concentration profiles, establishing a solid theoretical and computational bridge between nonlocal phase-field models and peridynamic corrosion frameworks.

This work lays a rigorous and versatile foundation for building a comprehensive nonlocal toolset. A key goal is to extend this unified theory to encompass coupled electrochemical transport, mechanical deformation, and multicomponent chemistry. In particular, integrating our recent development of a nonlocal Nernst--Planck--Poisson (NNPP) system for multion transport with this KKS-style phase-field model—while also incorporating elastic and plastic mechanics in a nonlocal setting—would enable the simulation of moving boundaries, pH-dependent reactions, and mechanical stresses within a single consistent theory. Such an extended model could capture not only peridynamic-style pit growth but also coupled electromigration, stress-affected dissolution, and the evolution of a diffuse corrosion layer.

Future work will therefore focus on this integration, alongside pursuing rigorous error estimates for fully discrete schemes, extension to three-dimensional geometries, and validation against experimental data for multi-ion, pH-coupled corrosion (e.g., Mg-based biomaterials). We also plan to explore alternative nonlocal kernels—anisotropic or weighted by material microstructure—to enhance modeling flexibility. Ultimately, by unifying phase-field, peridynamic, and NNPP approaches, we aim to provide a comprehensive framework for simulating complex, multiscale corrosion processes across engineering and biomedical applications.

\section*{Acknowledgements}
This work was supported by the IDIR-Project (Digital Implant Research), a cooperation financed by Kiel University, University Hospital Schleswig-Holstein and Helmholtz Zentrum Hereon. Funded by the Deutsche Forschungsgemeinschaft (DFG, German Research Foundation) - 470246804.

\bibliography{ref.bib}
\bibliographystyle{siam}

\end{document}